\documentclass{amsart}
\usepackage{amsmath}
  \usepackage{paralist}
  \usepackage{graphics} %% add this and next lines if pictures should be in esp format
  \usepackage{epsfig} %For pictures: screened artwork should be set up with an 85 or 100 line screen
\usepackage{graphicx}  \usepackage{epstopdf}%This is to transfer .eps figure to .pdf figure; please compile your paper using PDFLeTex or PDFTeXify.
 \usepackage[colorlinks=true]{hyperref}
\hypersetup{urlcolor=blue, citecolor=red}
\newtheorem{theorem}{Theorem}[section]

\newtheorem{lemma}[theorem]{Lemma}
\newtheorem{proposition}{Proposition}

\theoremstyle{definition}

\newcommand{\RR}{\mathbb R}
\newcommand{\NN}{\mathbb N}

\DeclareMathOperator*{\essinf}{ess\,inf}

\title[Perturbations of Robin eigenvalue problems] %Use the shortened version of the full title
      {Positive solutions for perturbations of the Robin eigenvalue problem plus an indefinite potential}

\author[N. S. Papageorgiou, V. D. R\u adulescu and D. D. Repov\v s]{}

\subjclass[2010]{Primary: 35J20; Secondary: 35J60.}
 \keywords{Indefinite and unbounded potential, Robin eigenvalue problem, sublinear perturbation, superlinear perturbation, maximum principle, positive solution, minimal positive solution.}

 \email{npapg@math.ntua.gr}
 \email{vicentiu.radulescu@imfm.si}
 \email{dusan.repovs@guest.arnes.si}

\begin{document}
\maketitle

% Enter the first author's name and address:
\centerline{\scshape Nikolaos S. Papageorgiou}
\medskip
{\footnotesize
% please put the address of the first author
 \centerline{Department of Mathematics}
   \centerline{National Technical University,
				Zografou Campus}
   \centerline{Athens 15780, Greece}
} % Do not forget to end the {\footnotesize by the sign }

\medskip

\centerline{\scshape Vicen\c tiu D. R\u adulescu}
\medskip
{\footnotesize
 % please put the address of the second  and third author
 \centerline{ Department of Mathematics, Faculty of Sciences}
   \centerline{King Abdulaziz University, P.O.
Box 80203, Jeddah 21589, Saudi Arabia}
   \centerline{Department of Mathematics, University of Craiova}
   \centerline{ Street A.I. Cuza No 13,
          200585 Craiova, Romania}
}

\medskip
\centerline{\scshape Du\v san D. Repov\v s}
\medskip
{\footnotesize
 % please put the address of the second  and third author
 \centerline{Faculty of Education and Faculty of Mathematics and Physics}
   \centerline{University of Ljubljana, Kardeljeva plo\v{s}\v{c}ad 16}
   \centerline{SI-1000 Ljubljana, Slovenia}
}

\bigskip

%The abstract of your paper
\begin{abstract}
We study perturbations of the eigenvalue problem for the negative
Laplacian plus an indefinite and unbounded potential and Robin
boundary condition. First we consider the case of a sublinear
perturbation and then of a superlinear perturbation. For the first
case we show that for $\lambda<\widehat{\lambda}_{1}$
($\widehat{\lambda}_{1}$ being the principal eigenvalue) there is
one positive solution which is unique under additional conditions on
the perturbation term. For $\lambda\geq\widehat{\lambda}_{1}$ there
are no positive solutions. In the superlinear case, for
$\lambda<\widehat{\lambda}_{1}$ we have at least two positive
solutions and for $\lambda\geq\widehat{\lambda}_{1}$ there are no
positive solutions. For both cases we establish the existence of a
minimal positive solution $\bar{u}_{\lambda}$ and we investigate the
properties of the map $\lambda\mapsto\bar{u}_{\lambda}$.
\end{abstract}

%The title of your section 1
\section{Introduction}\label{sec1}

Let $\Omega\subseteq\RR^{N}$ be a bounded domain with a
$C^{2}$-boundary $\partial\Omega$. In this paper we study the
following semilinear parametric Robin problem with an indefinite and
unbounded potential $\xi(z)$:
$$\left\{\begin{array}{ll}
-\Delta u(z)+\xi(z)u(z)=\lambda u(z)+f(z,u(z))\ \ \mbox{in}\ \Omega,\\ [0.3cm]
\displaystyle\frac{\partial u}{\partial n}+\beta(z)u=0\ \ \mbox{on}\ \partial\Omega,\ u\geq0.
\end{array}\right\}
\hspace{2cm} (P_{\lambda})$$

In this problem $\lambda\in\RR$ is a parameter and $\xi(\cdot)$ is a
potential function which is indefinite (that is, $\xi(\cdot)$ is
sign changing) and unbounded from below. We can think of
($P_{\lambda}$) as a perturbation of the standard eigenvalue
problem for the differential operator $u\mapsto -\Delta
u+\xi(z)u$ with Robin boundary condition. We look for positive
solutions and consider two cases: a sublinear perturbation
$f(z,\cdot)$ and a superlinear perturbation $f(z,\cdot)$. For
both cases we determine the dependence of the set of positive
solutions as the parameter $\lambda\in\RR$ varies.

We mention that the standard eigenvalue problems for the Robin
Laplacian have recently been studied by D'Agui, Marano \& Papageorgiou
\cite{4} and by Papageorgiou \& R\u{a}dulescu \cite{11}. Additional
existence and multiplicity results for parametric Robin and Neumann
problems can be found in the works of Papageorgiou \& R\u{a}dulescu
\cite{12, 13}.

Let $\widehat{\lambda}_{1}\in\RR$ be the principal eigenvalue of the
operator $u\mapsto -\Delta u+\xi(z)u$ with Robin boundary
condition. In the {\it sublinear} case (that is, when $f(z,\cdot)$ is
sublinear near $+\infty$) we show that for
$\lambda\geq\widehat{\lambda}_{1}$ problem ($P_{\lambda}$) has no
positive solution, whereas for $\lambda<\widehat{\lambda}_{1}$ problem
($P_{\lambda}$) has at least one positive solution. In fact, we show
that under an additional monotonicity condition on the quotient
$x\mapsto\frac{f(z,x)}{x}$ on $(0,+\infty)$, this positive
solution is unique. In the {\it superlinear} case (that is, when
$f(z,\cdot)$ is superlinear near $+\infty$) the situation changes and
uniqueness of the solution fails. In fact, we show that the equation
exhibits a kind of bifurcation phenomenon. Namely, for
$\lambda\geq\widehat{\lambda}_{1}$ problem ($P_{\lambda}$) has no
positive solution, whereas for $\lambda<\widehat{\lambda}_{1}$ it has
at least two positive solutions. For both cases, we show that the
problem has a minimal (that is, smallest) positive solution
$\bar{u}_{\lambda}$ and we determine the monotonicity and continuity
properties of the map $\lambda\mapsto\bar{u}_{\lambda}$.

Our approach is variational, based on the critical point theory,
together with suitable truncation and perturbation techniques. In
the next section, for the convenience of the reader, we recall the
main mathematical tools which will be used in the sequel.

\section{Mathematical background}

Let $X$ be a Banach space and $X^{*}$ be its topological dual. By
$\langle\cdot,\cdot\rangle$ we denote the duality brackets for the
pair $(X^{*},X)$. Given $\varphi\in C^{1}(X,\RR)$, we say that
$\varphi$ satisfies the {\it Cerami condition} (the {\it
C-condition} for short), if the following is true:%\\

{\center Every sequence $\{u_{n}\}_{n\geq1}\subseteq X$ such that $\{\varphi(u_{n})\}_{n\geq1}\subseteq\RR$ is bounded and\\ \vspace{0.2cm}
$(1+\|u_{n}\|)\varphi'(u_{n})\rightarrow0$ in $X^{*}$ as $n\rightarrow\infty$,\\ \vspace{0.2cm}
\hspace{1.9cm} admits a strongly convergent subsequence.}\vspace{0.2cm}

This compactness-type condition on $\varphi$ replaces the local compactness of the ambient space $X$ (in most applications $X$ is infinite dimensional and so it is not locally compact). It leads to a deformation theorem, from which one can derive the minimax theory of the critical values of $\varphi$. A central result of this theory, is the so-called {\it mountain pass theorem}, which we state here in a slightly more general version (see, for example, Gasinski \& Papageorgiou \cite{6}).

\begin{theorem}\label{Theorem 1}
Assume that $\varphi\in C^{1}(X,\RR)$ satisfies the C-condition, $u_{0}, u_{1}\in X$, $\|u_{1}-u_{0}\|>\rho$,
$$\max\{\varphi(u_{0}), \varphi(u_{1})\}<\inf[\varphi(u):\, \|u-u_0\|=\rho]=m_{\rho}$$
and $c=\inf_{\gamma\in\Gamma}\limits\max_{0\leq t\leq1}\limits\varphi(\gamma(t))$, where $\Gamma=\left\{\gamma\in C([0,1], X): \gamma(0)=u_{0}, \gamma(1)=u_{1}\right\}$. Then $c\geq m_{\rho}$ and $c$ is a critical value of $\varphi$.
\end{theorem}

The analysis of problem ($P_{\lambda}$) will make use of the Sobolev space $H^{1}(\Omega)$. This is a Hilbert space with inner product

$$(u,h)_{H^{1}(\Omega)}=\int_{\Omega}uh\,dz+\int_{\Omega}(Du,Dh)_{\RR^{N}}\,dz\ \mbox{for all}\ u,h\in H^{1}(\Omega).$$
The norm corresponding to this inner product, is given by
$$\|u\|=\left[\|u\|_{2}^{2}+\|Du\|_{2}^{2}\right]^{1/2}\ \mbox{for all}\ u\in H^{1}(\Omega).$$

In addition, we will also use the Banach spaces $C^{1}(\overline{\Omega})$ and $L^{\tau}(\partial\Omega)$, $1\leq\tau\leq\infty$.

The space $C^{1}(\overline\Omega)$ is an ordered Banach space with order (positive) cone
$$C_{+}=\{u\in C^{1}(\overline\Omega):\, u(z)\geq0\ \mbox{for all}\ z\in\overline\Omega\}.$$

This cone has a nonempty interior, given by
$$\textrm{int}\,C_{+}=\{u\in C_{+}:\, u(z)>0\ \mbox{for all}\ z\in\overline\Omega\}.$$

On $\partial\Omega$ we consider the $(N-1)$-dimensional Hausdorff (surface) measure $\sigma(\cdot)$. Using this measure we can define the Lebesgue spaces $L^{\tau}(\partial\Omega)$ $(1\leq\tau\leq\infty)$ in the usual way. The theory of Sobolev spaces says that there exists a unique continuous linear map $\gamma_{0}: H^{1}(\Omega)\rightarrow L^{2}(\partial\Omega)$, known as the {\it trace map}, such that
$$\gamma_{0}(u)=u|_{\partial\Omega}\ \mbox{for all}\ u\in H^{1}(\Omega)\cap C(\overline\Omega).$$

 Therefore we can interpret $\gamma_{0}(u)$ as representing the {\it boundary values} of $u\in H^{1}(\Omega)$. From the general theory of Sobolev spaces, we know that
$$\textrm{im}\,\gamma_{0}=H^{\frac{1}{2},2}(\partial\Omega) \ \ \mbox{and}\ \ \textrm{ker}\,\gamma_{0}=H^{1}_{0}(\Omega).$$

In addition, we know that $\gamma_{0}(\cdot)$ is compact from $H^{1}(\Omega)$ into $L^{\tau}(\partial\Omega)$ with $\tau\in\left[1,\frac{2(N-1)}{N-2}\right)$ if $N\geq3$ and $\tau\in[1,+\infty)$ if $N=1,2$.

In what follows, for the sake of notational simplicity, we shall drop the use of the map $\gamma_{0}$. All restrictions of functions $u\in H^{1}(\Omega)$ on the boundary $\partial\Omega$, are understood in the sense of traces.

Suppose that $f_{0}:\Omega\times\RR\rightarrow\RR$ is a Carath\'{e}odory function (that is, for all $x\in\RR$, $z\mapsto f(z,x)$ is measurable and for a.a. $z\in\Omega$, $x\mapsto f(z,x)$ is continuous) which has subcritical growth in the $x\in\RR$ variable. Hence
$$|f_{0}(z,x)|\leq a_{0}(z)(1+|x|^{r-1})\ \mbox{for a.a.}\ z\in\Omega,\ \mbox{all}\ x\in\RR,$$
with $a_{0}\in L^{\infty}(\Omega)_{+}$, $2\leq r<2^{*}=\left\{ \begin{array}{ll}
   \frac{2N}{N-2}\ \ \mbox{if}\ N\geq3\\
   +\infty\ \ \mbox{if}\ N=1,2 \end{array}\right.$ (the critical Sobolev exponent). We set $F_{0}(z,x)=\int_{0}^{x}f_{0}(z,s)\,ds$ and consider the $C^{1}$-functional $\varphi_{0}:H^{1}(\Omega)\rightarrow\RR$ defined by
$$\varphi_{0}(u)=\frac{1}{2}\gamma(u)-\int_{\Omega}F_{0}(z,u(z))\, dz\ \ \mbox{for all}\ u\in H^{1}(\Omega),$$
where $\gamma:H^{1}(\Omega)\rightarrow\RR$ is the $C^{1}$-functional defined by
$$\gamma(u)=\|Du\|_{2}^{2}+\int_{\Omega}\xi(z)u^{2}\, dz+\int_{\partial\Omega}\beta(z)u^{2}\, d\sigma\ \ \mbox{for all}\ u\in H^{1}(\Omega).$$

The next result relates local minimizers of $\varphi_{0}$ in $C^{1}(\overline\Omega)$ and in $H^{1}(\Omega)$, respectively. It is an outgrowth of the regularity theory for such problems (see Wang \cite{14}) and a more general version of it (with proof) can be found in Papageorgiou \& R\u{a}dulescu \cite{11}. We mention that the first result of this kind for the space $H^{1}_{0}(\Omega)$ (Dirichlet problems) and $\xi\equiv0$, was proved by Brezis \& Nirenberg \cite{3}. Our conditions on the potential $\xi(\cdot)$ and on the boundary coefficient $\beta(\cdot)$ are:\\ [0.3cm]
$H(\xi)$: $\xi\in L^{s}(\Omega)$ with $s>N$ if $N\geq3$ and $s>1$ if $N=1,2$\\ [0.2cm]
$H(\beta)$: $\beta\in W^{1,\infty}(\partial\Omega)$ with $\beta(z)\geq0$ for all $z\in\partial\Omega$

\vspace*{4pt}\noindent\textbf{Remark.} If $\beta\equiv0$, then we recover the Neumann problem. Hence our present paper includes the Neumann problems as a special case.

\begin{proposition}\label{Proposition 2}
 Assume that hypotheses $H(\xi)$, $H(\beta)$ hold and $u_{0}\in H^{1}(\Omega)$ is a local $C^{1}(\overline\Omega)$-minimizer, that is, there exists $\rho_{1}>0$ such that  $\varphi_{0}(u_{0})\leq\varphi_{0}(u_{0}+h)$ for all $h\in C^{1}(\overline\Omega)$ with $\|h\|_{C^{1}(\overline\Omega)}\leq\rho_{1}$. Then $u_{0}\in C^{1,\alpha}(\overline\Omega)$ for some $\alpha\in(0,1)$ and $u_0$ is a local $H^{1}(\Omega)$-minimizer of $\varphi_{0}$, that is, there exists $\rho_{2}>0$ such that  $\varphi_{0}(u_{0})\leq\varphi_{0}(u_{0}+h)$ for all $h\in H^{1}(\Omega)$ with $\|h\|\leq{\rho_{2}}$.
\end{proposition}

Next we recall some basic facts concerning the spectrum of the differential operator $u\mapsto -\Delta u+\xi(z)u$ with Robin boundary condition. So we consider the following linear eigenvalue problem
\begin{eqnarray}\label{1}
\left\{\begin{array}{ll}
-\Delta u(z)+\xi(z)u(z)=\widehat{\lambda} u(z)\ \ \mbox{in}\ \Omega,\\ [0.3cm]
\displaystyle\frac{\partial u}{\partial n}+\beta(z)u=0\ \ \mbox{on}\ \partial\Omega.
\end{array}\right\}
\end{eqnarray}

By an eigenvalue we mean a $\widehat{\lambda}\in\RR$ for which problem \eqref{1} has a nontrivial solution $\hat{u}\in H^{1}(\Omega)$, called an eigenfunction corresponding to $\widehat\lambda$. Using the spectral theorem for compact self-adjoint operators on a Hilbert space, we know that the spectrum of \eqref{1} consists of a sequence $\{\widehat\lambda_{k}\}_{k\geq1}$ of distinct eigenvalues such that  $\widehat\lambda_{k}\rightarrow+\infty$ (see \cite{4, 11}). The first eigenvalue $\widehat\lambda_{1}\in\RR$ has the following properties:
\begin{itemize}
  \item $\widehat\lambda_{1}$ is simple with eigenfunctions of constant sign;
  \item we have \begin{eqnarray}\label{2}\widehat\lambda_{1}=\inf\left[\frac{\gamma(u)}{\|u\|_{2}^{2}}:\, u\in H^{1}(\Omega), u\neq0\right].\end{eqnarray}
\end{itemize}

The infimum in \eqref{2} is realized on the corresponding one dimensional eigenspace. Let $\hat{u}_{1}$ be the $L^{2}$-normalized (that is, $\|\hat{u}_{1}\|_2=1$) positive eigenfunction corresponding to $\widehat\lambda_{1}$. If hypotheses $H(\xi)$ and $H(\beta)$ hold, then $\hat{u}_{1}\in C^1(\overline\Omega)$ (see Wang \cite{14}) and by Harnack's inequality (see Gasinski \& Papageorgiou \cite[p. 731]{6}), we have $\hat u_{1}(z)>0$ for all $z\in\Omega$. Moreover, if in addition we assume that $\xi^{+}\in L^{\infty}(\Omega)$, then $\hat u_{1}\in\textrm{int}\, C_{+}$ (via the strong maximum principle, see for example Gasinski \& Papageorgiou \cite[p. 738]{6}).

As a consequence of these properties, we obtain the following simple lemma.

\begin{lemma}\label{Lemma 3}
If hypotheses $H(\xi)$, $ H(\beta)$ hold, $\vartheta\in L^{\infty}(\Omega)$, $\vartheta(z)\leq\widehat{\lambda}_{1}$ for a.a. $z\in\Omega$ and $\vartheta\not\equiv\widehat\lambda_1$, then there exists $\hat{c}>0$ such that
$$J(u)=\gamma(u)-\int_{\Omega}\vartheta(z)u^{2}\, dz\geq\hat c\|u\|^2 \ \ \mbox{for all}\ \ u\in H^{1}(\Omega).$$
\end{lemma}

\begin{proof}
From the variational characterization of $\widehat\lambda_{1}$ (see \eqref{2}), we get that $J\geq0$. Suppose that the lemma is not true. Exploiting the $2$-homogeneity of $J(\cdot)$, we can then find $\{u_{n}\}_{n\geq1}\subseteq H^{1}(\Omega)$ such that
\begin{eqnarray}\label{3}
\|u_{n}\|=1\ \mbox{for all}\ n\in\NN\ \mbox{and}\ J(u_{n})\downarrow0\ \mbox{as}\ n\rightarrow\infty.
\end{eqnarray}

We may assume that
\begin{eqnarray}\label{4}
u_{n}\xrightarrow{w}u\ \mbox{in}\ H^{1}(\Omega)\ \mbox{and}\ u_{n}\rightarrow u\ \mbox{in}\ L^{2}(\Omega)\ \mbox{and in}\ L^{2}(\partial\Omega).
\end{eqnarray}

It follows from \eqref{3} and \eqref{4} that
\begin{align}
{}&J(u)\leq0,\nonumber \\
\label{5}\Rightarrow\, & \gamma(u)\leq\int_{\Omega}\vartheta(z)u^{2}\,dz\leq\widehat\lambda_{1}\|u\|_{2}^{2}\\
\Rightarrow\, & \gamma(u)=\widehat\lambda_{1}\|u\|_{2}^{2}\ \ (\mbox{see}\ \eqref{2}),\nonumber\\
\Rightarrow\, & u=\chi\hat{u}_{1}\ \mbox{with}\ \chi\in\RR. \nonumber
\end{align}
If $\chi=0$, then $u=0$ and so
\begin{align*}
{}&\|Du_{n}\|_{2}\rightarrow0\ (\mbox{see}\ \eqref{3}),\\
\Rightarrow\, & u_{n}\rightarrow0\ \mbox{in}\ H^{1}(\Omega)\ (\mbox{see}\ \eqref{4}),
\end{align*}
which contradicts the fact that $\|u_{n}\|=1$ for all $n\in\NN$ (see \eqref{3}).

Hence $\chi\neq0$. Then $u(z)\neq0$ for all $z\in\Omega$ and by \eqref{5} we have
$$\|Du\|_{2}^{2}<\widehat\lambda_{1}\|u\|_{2}^{2},$$
which contradicts \eqref{2}.
\end{proof}

We observe that there exists $\mu>0$ such that
\begin{eqnarray}\label{6}
\gamma(u)+\mu\|u\|_{2}^{2}\geq c_{0}\|u\|^{2}\ \mbox{for some}\ c_{0}>0,\ \mbox{all}\ u\in H^{1}(\Omega)
\end{eqnarray}
(see \cite{4, 13}).

To see \eqref{6}, we argue by contradiction. So, suppose that the inequality is not true. We can find $\{u_n\}_{n\geq 1}\subseteq H^1(\Omega)$ such that
$$\gamma (u_n)+n\,\| u_n\|^2_2<\frac 1n\, \|u_n\|\quad\mbox{for all}\ n\in\NN\,.$$
Set $y_n=\displaystyle\frac{u_n}{\|u_n\|}$, $n\in\NN$. Then $\|y_n\|=1$ for all $n\in\NN$ and so we may assume that
$$y_{n}\xrightarrow{w}y\ \mbox{in}\ H^1(\Omega) \ \mbox{and}\ y_{n}\rightarrow y\ \mbox{in}\ L^2(\Omega)\ \mbox{and in}\ L^2(\partial\Omega).$$
Exploiting the sequential weak lower semicontinuity of $\gamma (\,\cdot\,)$ we see that $y=0$ and $n\|y_n\|^2_2\rightarrow 0$. Therefore $\| y_n\|_2\rightarrow 0$. Finally we can say that
\begin{align}
&0\leq \liminf_{n\rightarrow\infty}\gamma (y_n)\leq\limsup_{n\rightarrow\infty}\gamma (y_n)\leq
\lim_{n\rightarrow\infty}\left[ \frac 1n-n\|y_n\|^2_2\right]=0,\nonumber \\
\nonumber\Rightarrow\, 0 & =\lim_{n\rightarrow\infty}\gamma (y_n)=\lim_{n\rightarrow\infty}\|Dy_n\|^2_2\ (\mbox{recall}\ y=0)\\
\nonumber\Rightarrow\,\,\,\,\, & y_n\rightarrow 0\ \mbox{in $H^1(\Omega)$, a contradiction to the fact that $\|y_n\|=1$}.
\end{align}

We say that a Banach space $X$ has the \textit{Kadec-Klee property} if the following is true:
\begin{eqnarray}\label{7}
u_{n}\xrightarrow{w}u\ \mbox{in}\ X\ \mbox{and}\ \|u_{n}\|\rightarrow\|u\|\Rightarrow u_{n}\rightarrow u\ \mbox{in}\ X.
\end{eqnarray}

As a consequence of the parallelogram law, we see that every Hilbert space has the Kadec-Klee property.

Given $x\in\RR$, we set $x^{\pm}=\max\{\pm x,0\}$. Then for every $u\in H^{1}(\Omega)$ we define
$$u^{\pm}(\cdot)=u(\cdot)^{\pm}.$$
We know that
$$u^{\pm}\in H^{1}(\Omega),\; u=u^{+}-u^{-},\; |u|=u^{+}+u^{-}\ \mbox{for all}\ u\in H^{1}(\Omega).$$

By $|\cdot|_{N}$ we denote the Lebesgue measure on $\RR^{N}$ and given a measurable function $h:\Omega\times\RR\rightarrow\RR$ (for example, a Carath\'{e}odory function), we set
$$N_{h}(u)(\cdot)=h(\cdot, u(\cdot))\ \mbox{for all}\ u\in H^{1}(\Omega)$$
(the Nemytskii or superposition map corresponding to the function $h$).

Given $\varphi\in C^{1}(X,\RR)$ ($X$ a Banach space), we set
$$K_{\varphi}=\{u\in X:\; \varphi'(u)=0\}\ (\mbox{the critical set of $\varphi$}).$$

Finally, by $A\in\mathcal{L}(H^{1}(\Omega), H^{1}(\Omega)^{*})$ we denote the linear operator defined by
$$\langle A(u),h\rangle=\int_{\Omega}(Du,Dh)_{\RR^{N}}\, dz\ \mbox{for all}\ u,h\in H^{1}(\Omega).$$

\section{The sublinear case}\label{sec2}

In this section we examine problem ($P_{\lambda}$) under the hypothesis that the perturbation term $f(z,\cdot)$ is sublinear near $+\infty$. More precisely, our conditions on the nonlinearity $f(z,x)$ are the following: \vspace{0.2cm}\\
$H_{1}$:\ \ $f:\Omega\times\RR\rightarrow\RR$ is a Carath\'{e}odory function such that  for a.a. $z\in\Omega$, $f(z,0)=0$, $f(z,x)>0$ for all $x>0$ and
\begin{enumerate}
\item [(i)] for every $\rho>0$, there exists $a_{\rho}\in L^{\infty}(\Omega)_{+}$ such that  $f(z,x)\leq a_{\rho}(z)$ for a.a. $z\in\Omega$, all $0\leq x\leq\rho$;\\
\item [(ii)] $\lim_{x\rightarrow+\infty}\limits\frac{f(z,x)}{x}=0$ uniformly for a.a. $z\in\Omega$;\\
\item [(iii)] there exist $\delta>0$ and $q\in(1,2)$ such that
$c_{1}x^{q-1}\leq f(z,x)\ \mbox{for a.a.}\ z\in\Omega,\ \mbox{all}\ 0\leq x\leq\delta$.
\end{enumerate}

\vspace*{4pt}\noindent\textbf{Remarks.}
Since we are looking for positive solutions and all the above hypotheses concern the positive semiaxis $\RR_{+}=[0,+\infty)$,  we may assume without any loss of generality that $f(z,x)=0$ for a.a. $z\in\Omega$, all $x\leq0$. Hypothesis $H_{1}$(ii) implies that $f(z,\cdot)$ is sublinear near $+\infty$. Hypothesis $H_{1}$(iii) says that there is a concave term near the origin.

\vspace*{4pt}\noindent\textbf{Examples.} The following functions satisfy hypotheses $H_{1}$. For the sake of simplicity we drop the $z$-dependence
$$f_{1}(x)=x^{q-1}\ \mbox{for all}\ x\geq0,\ \mbox{with}\ 1<q<2,$$
$$
f_{2}(x)=\left\{\begin{array}{ll}
-x^{q-1}\ln x & \mbox{if}\ x\in[0,1]\\ [0.3cm]
x^{s-1}-x^{p-1} & \mbox{if}\ 1<x
\end{array}\right. \mbox{with}\ 1<p<s<2,\; 1<q<2.
$$

We can improve the properties of the positive solutions of ($P_{\lambda}$), provided we strengthen hypothesis $H(\xi)$.\vspace{0.2cm}\\
$H(\xi)'$:\ \ $\xi\in L^{s}(\Omega)$ with $s>N$ if $N\geq3$, $s>1$ if $N=1,2$ and $\xi^{+}\in L^{\infty}(\Omega)$. \vspace{0.2cm}

We introduce the following two sets:
$$\mathcal{L}=\{\lambda\in\RR:\; \mbox{problem}\ (P_{\lambda})\; \mbox{admits a positive solution}\},$$
$$S(\lambda)=\{\mbox{the set of positive solutions for problem}\ (P_{\lambda})\}.$$

We start with a simple but useful observation concerning the solution set $S(\lambda)$. For this result the precise behavior of $f(z,\cdot)$ near $+\infty$ and near $0^{+}$ are irrelevant. We only need nonnegativity of $f(z,x)$ and subcritical growth in $x\in\RR$. Under these very general conditions, the result is also applicable in the superlinear case (see Section \ref{sec4}). The new hypotheses on the perturbation term $f(z,x)$ are the following:
\begin{enumerate}
\item [$\widehat{H}$:] $f:\Omega\times\RR\rightarrow\RR$ is a Carath\'{e}odory function such that  for a.a. $z\in\Omega$, $f(z,0)=0$, $f(z,x)\geq0$ for all $x\geq0$ and
$$f(z,x)\leq a(z)(1+x^{r-1})\ \mbox{for a.a.}\ z\in\Omega,\ \mbox{all}\ x\geq0,$$
with $a\in L^{\infty}(\Omega)_{+}$, $2\leq r<2^{*}$.
\end{enumerate}

\begin{proposition}\label{Proposition 4}
If hypotheses $H(\xi)$ (resp. $H(\xi)'$), $H(\beta)$, $\widehat{H}$ hold, then for all $\lambda\in\RR$, $S(\lambda)\subseteq C_{+}\setminus\{0\}$ (resp. $S(\lambda)\subseteq \textrm{int}\, C_{+}$) (possibly empty).
\end{proposition}

\begin{proof}
Let $u\in S(\lambda)$. Then we have
\begin{eqnarray}\label{8}\left\{\begin{array}{ll}
-\Delta u(z)+\xi(z)u(z)=\lambda u(z)+f(z,u(z))\ \ \mbox{for a.a.}\ z\in\Omega,\vspace{0.2cm}\\
\displaystyle\frac{\partial u}{\partial n}+\beta(z)u=0\ \ \mbox{on}\ \partial\Omega
\end{array}\right\}
\end{eqnarray}
(see Papageorgiou \& R\u{a}dulescu \cite{11}).

We introduce the functions
$$
k_{\lambda}(z)=\left\{
\begin{array}{ll}
0 & \mbox{if}\ 0\leq u(z)\leq1\vspace{0.2cm}\\
\displaystyle\frac{f(z,u(z))}{u(z)}+(\lambda-\xi(z)) & \mbox{if}\ 1<u(z)
\end{array} \right.
$$
and
$$
\vartheta_{\lambda}(z)=\left\{
\begin{array}{ll}
f(z,u(z))+(\lambda-\xi(z))u(z) & \mbox{if}\ 0\leq u(z)\leq1\vspace{0.2cm}\\
0 & \mbox{if}\ 1<u(z).
\end{array} \right.
$$

Evidently, $\vartheta_{\lambda}\in L^{s}(\Omega)$ (see hypotheses $H(\xi)$ and $\widehat{H}$). Also
$$|k_{\lambda}(z)|\leq c_{2}(1+u(z)^{r-2})+(\lambda-\xi(z))\ \mbox{for a.a.}\ z\in\Omega,\ \mbox{some}\ c_{2}>0.$$
Note that if $N\geq3$ (the cases $N=1,2$ are straightforward), then
\begin{align*}
(r-2)\frac{N}{2}&<\left(\frac{2N}{N-2}-2\right)\frac{N}{2}\ (\mbox{since}\ r<2^{*})\\
&=\frac{2N}{N-2}=2^{*}.
\end{align*}
Since $u\in H^{1}(\Omega)$, by the Sobolev embedding theorem, we have
\begin{align*}
& u^{(r-2)\frac{N}{2}}\in L^{1}(\Omega),\\
\Rightarrow\, & k_{\lambda}\in L^{\frac{N}{2}}(\Omega)\ (\mbox{see hypothesis}\ H(\xi)).
\end{align*}
We rewrite \eqref{8} as follows
$$\left\{\begin{array}{ll}
-\Delta u(z)=k_\lambda (z)u(z)+\vartheta_{\lambda}(z)\ \ \mbox{for a.a.}\ z\in\Omega,\vspace{0.2cm}\\
\displaystyle\frac{\partial u}{\partial n}+\beta(z)u=0\ \ \mbox{on}\ \partial\Omega.
\end{array}\right\}
$$
 By Lemma 5.1 of Wang \cite{14} we have that $u\in L^{\infty}(\Omega)$. Using the Calderon-Zygmund estimates (see Lemma 5.2 of Wang \cite{14}), we obtain $u\in W^{2,s}(\Omega)$. Then the Sobolev embedding theorem implies $u\in C^{1,\alpha}(\overline\Omega)$ with $\alpha=1-\frac{N}{s}>0$. Therefore we have $u\in C_{+}\setminus\{0\}$. Suppose that $H(\xi)'$ holds. From \eqref{8} we obtain
\begin{align*}
& \Delta u(z)\leq(\xi(z)-\lambda)u(z)\ \mbox{for a.a.}\ z\in\Omega \ (\mbox{see hypotheses}\ \widehat{H})\vspace{0.2cm}\\
&{}\hspace{1.03cm} \leq(\|\xi^{+}\|_{\infty}+|\lambda|)u(z)\ \mbox{for a.a.}\ z\in\Omega\ (\mbox{see hypothesis}\ H(\xi)')\vspace{0.5cm}\\
\Rightarrow\, & u\in\textrm{int}\, C_{+}\vspace{0.2cm}\\
& {}\ \hspace{1cm}(\mbox{by the strong maximum principle, see \cite[p. 738]{6}}).
\end{align*}

Thus, we have proved that when $H(\xi)'$ holds, then $S(\lambda)\subseteq\textrm{int}\, C_{+}$ for all $\lambda\in\RR$.
\end{proof}

Next, we show that for every $\lambda\geq\widehat{\lambda}_{1}$ problem ($P_{\lambda}$)  has no positive solutions (that is, $S(\lambda)=\emptyset$ for all $\lambda\geq\widehat{\lambda}_{1}$).

\begin{proposition}\label{Proposition 5}
If hypotheses $H(\xi)'$, $H(\beta)$, $H_{1}$ hold and $\lambda\geq\widehat{\lambda}_{1}$, then $S(\lambda)=\emptyset$.
\end{proposition}

\begin{proof}
Let $\lambda\geq\widehat{\lambda}_{1}$ and suppose that $S(\lambda)\neq\emptyset$. Let $u\in S(\lambda)$. Due to Proposition \ref{4}, we know that $u\in\textrm{int}\, C_{+}$. Let $v\in\textrm{int}\, C_{+}$ and consider the function
$$R(v,u)(z)=|Dv(z)|^{2}-\left(Du(z),D\left(\frac{v^{2}}{u}\right)(z)\right)_{\RR^{N}}.$$
From Picone's identity (see, for example, Motreanu, Motreanu \& Papageorgiou \cite[p. 255]{9}), we derive
\begin{align*}
& 0\leq R(v,u)(z)\ \mbox{for a.a.}\ z\in\Omega,\vspace{0.2cm}\\
\Rightarrow\, & 0\leq \int_{\Omega}R(v,u)\,dz\vspace{0.2cm}\\
&{}\hspace{0.2cm}=\|Dv\|_{2}^{2}-\int_{\Omega}\left(Du,D\left(\frac{v^{2}}{u}\right)\right)_{\RR^{N}}\, dz\vspace{0.3cm}\\
&{}\hspace{0.2cm}=\|Dv\|_{2}^{2}-\int_{\Omega}(-\Delta u)\left(\frac{v^2}{u}\right)\, dz+\int_{\partial\Omega}\beta(z)u\,\frac{v^2}{u}\, d\sigma\vspace{0.3cm}\\
&\hspace{0.7cm}(\mbox{using Green's identity, see Gasinski \& Papageorgiou \cite[p. 210]{6}})\vspace{0.3cm}\\
&{}\hspace{0.2cm}=\|Dv\|_{2}^{2}-\int_{\Omega}(\lambda-\xi(z))u\left(\frac{v^2}{u}\right)dz-\int_{\Omega}f(z,u)\frac{v^2}{u}\,dz
+\int_{\partial\Omega}\beta(z)v^2\,d\sigma\vspace{0.3cm}\\
&{}\hspace{0.2cm}<\gamma(v)-\lambda\|v\|_2^2\ \ \left(\mbox{since}\ f(z,u)\frac{v^{2}}{u}>0\ \mbox{a.a.}\ z\in\Omega\right).
\end{align*}
Let $v=\hat{u}_{1}\in\textrm{int}\, C_{+}$. Then $0<\gamma(\hat{u}_{1})-\lambda=\widehat\lambda_1-\lambda\leq0$ (recall $\|\hat u_1\|_2=1$), a contradiction. Therefore for all $\lambda\geq\widehat\lambda_1$, we have $S(\lambda)=\emptyset$.
\end{proof}

Next, we show that for $\lambda<\widehat\lambda_1$ there exist positive solutions.

\begin{proposition}\label{Proposition 6}
If hypotheses $H(\xi)'$, $H(\beta)$, $H_{1}$ hold and $\lambda<\widehat\lambda_1$, then $S(\lambda)\neq\emptyset$.
\end{proposition}

\begin{proof}
Let $\mu>0$ be as in \eqref{6} and consider the Carath\'{e}odory function $g_{\lambda}:\Omega\times\RR\rightarrow\RR$ defined by
$$
g_{\lambda}(z,x)=\left\{
\begin{array}{ll}
0 & \mbox{if}\ x\leq0 \\
(\lambda+\mu)x+f(z,x) & \mbox{if}\ 0<x.
\end{array} \right.
$$
We set $G_{\lambda}(z,x)=\int_{0}^{x}g_{\lambda}(z,s)\,ds$ and consider the $C^{1}$-functional $\varphi_{\lambda}:H^{1}(\Omega)\rightarrow\RR$ defined by
$$\varphi_{\lambda}(u)=\frac{1}{2}\gamma(u)+\frac{\mu}{2}\|u\|_{2}^{2}-\int_{\Omega}G_{\lambda}(z,u)dz\ \mbox{for all}\ u\in H^1(\Omega).$$
Hypotheses $H_{1}$(i),(ii) imply that given $\varepsilon>0$, we can find $c_{3}=c_{3}(\varepsilon)>0$ such that
\begin{eqnarray}\label{9}
F(z,x)\leq\frac{\varepsilon}{2}x^2+c_{3}\ \mbox{for a.a.}\ z\in\Omega,\ \mbox{all}\ x\geq0.
\end{eqnarray}
Using \eqref{9}, we obtain
\begin{align}
\varphi_{\lambda}(u)&\geq\frac{1}{2}\gamma(u)-\frac{\lambda+\varepsilon}{2}\|u^{+}\|_{2}^{2}-c_{3}|\Omega|_{N}\nonumber\\\vspace{0.3cm}
\label{10}&\geq\frac{1}{2}\gamma(u)-\frac{\lambda+\varepsilon}{2}\|u\|_2^2-c_3|\Omega|_N.
\end{align}
Choosing $\varepsilon\in(0,\widehat\lambda_{1}-\lambda)$ (recall $\lambda<\widehat\lambda_{1}$) and using Lemma \ref{Lemma 3}, from \eqref{10} we obtain
\begin{align*}
&\varphi_\lambda(u)\geq\frac{\hat{c}}{2}\|u\|^{2}-c_{3}|\Omega|_N,\\ \vspace{0.2cm}
\Rightarrow\, & \varphi_\lambda\ \mbox{is coercive}.
\end{align*}

Also, invoking the Sobolev embedding theorem and the compactness of the trace map, we see that $\varphi_\lambda$ is sequentially weakly lower semicontinuous. Therefore by the Weierstrass theorem, we can find $u_{\lambda}\in H^{1}(\Omega)$ such that
\begin{eqnarray}\label{11}
\varphi_\lambda(u_\lambda)=\inf\left[\varphi_\lambda(u):\; u\in H^1(\Omega)\right].
\end{eqnarray}
Let $t\in(0,1)$ be so small that  $t\hat{u}_1(z)\in(0,\delta]$ for all $z\in\overline\Omega$ (recall that $\hat{u}_1\in\textrm{int}\, C_{+}$ and note that $\delta>0$ is as in hypothesis $H_1$(iii)). Then
\begin{align}\label{12}
\varphi_{\lambda}(t\hat{u}_1)&\leq\frac{t^2}{2}\gamma(\hat{u}_1)-\frac{t^2}{2}\lambda-\frac{t^q}{q}\|\hat{u}_1\|_q^q\nonumber\\ \vspace{0.3cm}
&\hspace{0.3cm}\ (\mbox{see hypothesis}\ H_1\mbox{(iii) and recall}\ \|\hat{u}_1\|_2=1)\nonumber\\ \vspace{0.3cm}
&=\frac{t^2}{2}(\widehat\lambda_1-\lambda)-\frac{t^q}{q}\|\hat{u}_1\|_q^q.
\end{align}
Since $q<2$, choosing $t\in(0,1)$ even smaller if necessary, from \eqref{12} we obtain
\begin{align*}
&\varphi_\lambda(t\hat{u}_1)<0,\\ \vspace{0.3cm}
\Rightarrow\, &\varphi_{\lambda}(u_\lambda)<0=\varphi_\lambda(0)\ (\mbox{see \eqref{11}), hence}\ u_\lambda\neq0.
\end{align*}
By \eqref{11} we have
$$\varphi'_\lambda(u_\lambda)=0,$$
\begin{eqnarray}\label{13}
\Rightarrow\, \langle A(u_\lambda),h\rangle+\int_{\Omega}\xi(z)u_\lambda h\,dz+\int_{\partial\Omega}\beta(z)u_\lambda h\,d\sigma+\mu\int_\Omega u_\lambda h\,dz\\=\int_{\Omega}g_\lambda(z,u_\lambda)h\,dz
\mbox{for all}\ h\in H^{1}(\Omega).\nonumber
\end{eqnarray}
In \eqref{13} we choose $h=-u_{\lambda}^-\in H^{1}(\Omega)$. Then
\begin{align*}
&\gamma(u_\lambda^-)+\mu\|u_\lambda^-\|_2^2=0,\\ \vspace{0.3cm}
\Rightarrow\, & c_0\|u_\lambda^-\|^2\leq0\ (\mbox{see \eqref{6}}),\\ \vspace{0.3cm}
\Rightarrow\, & u_\lambda\geq0,\ u_\lambda\neq0.
\end{align*}
Thus, equation \eqref{13} becomes
\begin{align*}
&\langle A(u_\lambda),h\rangle+\int_{\Omega}\xi(z)u_\lambda h\,dz+\int_{\partial\Omega}\beta(z)u_\lambda h\,d\sigma=\int_{\Omega}(\lambda u_\lambda+f(z,u_\lambda))h\,dz\\ \vspace{0.3cm}
&{}\hspace{8.5cm}\mbox{for all}\ h\in H^1(\Omega),\\ \vspace{0.3cm}
\Rightarrow&-\Delta u_\lambda(z)+\xi(z)u_\lambda(z)=\lambda u_\lambda(z)+f(z,u_\lambda(z))\ \mbox{for a.a.}\ z\in\Omega,\\ \vspace{0.3cm}
&\frac{\partial u_\lambda}{\partial n}+\beta(z)u_{\lambda}=0\ \mbox{on}\ \partial\Omega\\ \vspace{0.3cm}
&{}\hspace{3cm}(\mbox{see Papageorgiou \& R\u{a}dulescu \cite{11}}),\\ \vspace{0.3cm}
\Rightarrow\, &u_\lambda\in S(\lambda)\subseteq\textrm{int}\, C_+\ (\mbox{see Proposition \ref{Proposition 4}) and so}\ S(\lambda)\neq\emptyset\ \mbox{for}\ \lambda<\widehat\lambda_1.
\end{align*}
\end{proof}

In fact, we can show that for every $\lambda<\widehat\lambda_1$ problem ($P_\lambda$) has a smallest positive solution. To this end note that given $\tau\in\left(\frac{2N}{N-1},2^{*}\right)$, because of hypotheses $H_{1}$(i),(iii), we can find $c_4(\lambda)>0$ with $\lambda\mapsto c_{4}(\lambda)$ bounded on bounded subsets of $\RR$ such that
\begin{eqnarray}\label{14}
\lambda x+f(z,x)\geq c_1 x^{q-1}-c_4(\lambda)x^{\tau-1}\ \mbox{for a.a.}\ z\in\Omega,\ \mbox{all}\ x\geq0.
\end{eqnarray}
This unilateral growth restriction on the reaction term of problem ($P_\lambda$) leads to the following auxiliary Robin problem:
$$\left\{\begin{array}{ll}
-\Delta u(z)+\xi(z)u(z)=c_1 u(z)^{q-1}-c_4(\lambda)u(z)^{\tau-1}\ \ \mbox{in}\ \Omega,\\ [0.3cm]
\displaystyle\frac{\partial u}{\partial n}+\beta(z)u=0\ \ \mbox{on}\ \partial\Omega,\ u\geq0.
\end{array}\right\}
\hspace{2cm} (A u_{\lambda})$$

\begin{proposition}\label{Proposition 7}
If hypotheses $H(\xi)'$, $H(\beta)$ hold and $\lambda\in\RR$, then problem ($A u_{\lambda}$) admits a unique positive solution $u_*^\lambda\in\textrm{int}\, C_+$.
\end{proposition}
\begin{proof}
First, we establish the existence of a positive solution for problem ($A u_\lambda$). So, we introduce the $C^1$-functional $\psi_{\lambda}:H^1(\Omega)\rightarrow\RR$ defined by
$$\psi_\lambda(u)=\frac{1}{2}\gamma(u)+\frac{\mu}{2}\|u^-\|_2^2-\frac{c_1}{q}\|u^+\|_q^q+\frac{c_4(\lambda)}{\tau}\|u^{+}\|_\tau^\tau\ \ \mbox{for all}\ u\in H^1(\Omega).$$
By Lemma \ref{Lemma 3}, we have
\begin{align}\label{15}
\psi_\lambda(u)&=\frac{1}{2}\gamma(u^-)+\frac{\mu}{2}\|u^-\|_2^2+\frac{1}{2}\gamma(u^{+})+\frac{c_4(\lambda)}{\tau}\|u^+\|_\tau^\tau
-\frac{c_1}{q}\|u^+\|_q^q\nonumber\\ \vspace{0.3cm}
&\geq\frac{\hat{c}}{2}\|u^-\|^2+
\frac{1}{2}\gamma(u^{+})+\frac{c_4(\lambda)}{\tau}\|u^+\|_\tau^\tau-\frac{c_1}{q}\|u^{+}\|_q^q\,.
\end{align}
Since for $N\geq3$, $s>N$, by the Sobolev embedding theorem, we have $u^2\in L^{s'}(\Omega)$ and so due to H\"{o}lder's inequality,
$$\left|\int_\Omega\xi(z)(u^+)^2\,dz\right|\leq\|\xi\|_s\|u^+\|_{2s'}^2.$$
Note that $s'<N'=\frac{N}{N-1}$ (recall that for $1\leq\tau<\infty$, $\tau'\in(1,+\infty]$ and $\frac{1}{\tau}+\frac{1}{\tau'}=1$), hence $2s'<\frac{2N}{N-1}<\tau$. Therefore
$$\left|\int_{\Omega}\xi(z)(u^{+})^{2}\,dz\right|\leq c_{5}\|u^{+}\|_{\tau}^{2}\ \mbox{for some}\ c_{5}>0.$$
Thus we have
\begin{align}\label{16}
&\frac{1}{2}\gamma(u^+)+\frac{c_4(\lambda)}{\tau}\|u^+\|_\tau^\tau-\frac{c_1}{q}\|u^{+}\|_q^q \nonumber\\ \vspace{0.3cm}
\geq&\frac{1}{2}\|D u^+\|_2^{2}+\frac{c_4(\lambda)}{\tau}\|u^+\|_\tau^\tau-c_5\|u^+\|_\tau^2-c_{6}\|u^+\|_\tau^q\ \ \mbox{for some}\ c_6>0 \nonumber\\ \vspace{0.3cm}
&\hspace{7cm}(\mbox{recall that}\ q<\tau)\nonumber\\ \vspace{0.3cm}
\geq& \frac{1}{2}\|D u^+\|_2^2+\frac{c_4(\lambda)}{\tau}\|u^+\|_\tau^\tau-c_7\left(\|u^+\|_{\tau}^2+1\right)\ \mbox{for some}\ c_7>0\nonumber\\ \vspace{0.3cm}
=&\frac{1}{2}\|D u^+\|_2^2+\left[\frac{c_4(\lambda)}{\tau}\|u^+\|_\tau^{\tau-2}-c_7\right]\|u^+\|_{\tau}^2-c_7.
\end{align}

We return to \eqref{15}, use \eqref{16} and recall that $y\mapsto\left[\|y\|_\tau^2+\|Dy\|_2^2\right]^{1/2}$ is an equivalent norm on the Sobolev space $H^1(\Omega)$ (see, for example, Gasinski \& Papageorgiou \cite[p. 227]{6}). So, from \eqref{16} we infer that $\psi_\lambda(\cdot)$ is coercive.

The cases $N=1,2$ are straightforward because
\begin{itemize}
\item if $N=1$, then $H^1(\Omega)\hookrightarrow C(\overline\Omega)$ (compactly);
\item if $N=2$, then $H^{1}(\Omega)\hookrightarrow L^\tau(\Omega)$ for all $\tau\in[1,+\infty)$ (compactly).
\end{itemize}
The Sobolev embedding theorem and the compactness of the trace map, imply that $\psi_\lambda$ is sequentially weakly lower semicontinuous. So, we can find $u_{*}^\lambda\in H^1(\Omega)$ such that
\begin{eqnarray}\label{17}
\psi_\lambda(u_{*}^\lambda)=\inf\left[\psi_\lambda(u):\; u\in H^{1}(\Omega)\right].
\end{eqnarray}
As before (see the proof of Proposition \ref{Proposition 6}), exploiting the fact that $q<2<\tau$, we obtain
\begin{align*}
&\psi_\lambda(u_{*}^\lambda)<0=\psi_\lambda(0),\\ \vspace{0.3cm}
\Rightarrow\, & u_{*}^\lambda\neq0.
\end{align*}
By \eqref{17} we have
\begin{align}\label{18}
& \psi_\lambda'(u_{*}^\lambda)=0, \nonumber\\ \vspace{0.3cm}
\Rightarrow\, & \langle A(u_{*}^\lambda),h\rangle+\int_{\Omega}\xi(z)u_{*}^\lambda h\,dz+\int_{\partial\Omega}\beta(z) u_{*}^\lambda h\, d\sigma-\mu\int_\Omega(u_{*}^\lambda)^- h\,dz \nonumber\\ \vspace{0.3cm}
& =c_1\int_{\Omega}((u_{*}^\lambda)^{+})^{q-1}h\,dz-c_{4}(\lambda)\int_{\Omega}((u_{*}^\lambda)^+)^{\tau-1}h\,dz\ \ \mbox{for all}\ h\in H^1(\Omega).
\end{align}
In \eqref{18} we choose $h=-(u_{*}^\lambda)^-\in H^1(\Omega)$. Then
\begin{align*}
&\gamma((u_{*}^\lambda)^-)+\mu\|(u_*^\lambda)^-\|_2^2=0, \\ \vspace{0.3cm}
\Rightarrow\, & \hat{c}\|(u_{*}^\lambda)^-\|^2\leq0\ \ (\mbox{see \eqref{6}}), \\ \vspace{0.3cm}
\Rightarrow\, & u_{*}^\lambda\geq0,\; u_{*}^\lambda\neq0.
\end{align*}

 Next, we infer from \eqref{18}  that $u_{*}^\lambda$ is a positive solution of ($A u_\lambda$). Again, using Lemmata 5.1 and 5.2 of Wang \cite{14}, we infer that $u_{*}^\lambda\in C_{+}\setminus\{0\}$. Moreover,  we have
\begin{align*}
&\Delta u_{*}^\lambda(z)\leq c_4(\lambda)u_{*}^\lambda(z)^{\tau-1}+\xi(z)u_{*}^\lambda(z)\ \ \mbox{for a.a.}\ z\in\Omega,\\ \vspace{0.3cm}
\Rightarrow\, & \Delta u_{*}^\lambda(z)\leq\left[c_{4}(\lambda)\|u_{*}^\lambda\|_\infty^{\tau-2}+\|\xi^{+}\|_{\infty}\right]u_{*}^\lambda(z)\ \ \mbox{for a.a.}\ z\in\Omega\\ \vspace{0.3cm}
&\hspace{5cm} (\mbox{see hypothesis}\ H(\xi)'),\\ \vspace{0.3cm}
\Rightarrow\, & u_{*}^\lambda\in\textrm{int}\, C_{+}\ (\mbox{by the strong maximum principle}).
\end{align*}

Next, we prove the uniqueness of this positive solution. So let $v_{*}^{\lambda}\in H^1(\Omega)$ be another positive solution for ($A u_\lambda$). As above, we show that $v_{*}^{\lambda}\in\textrm{int}\, C_{+}$. Let $t>0$ be the biggest real number such that $t v_{*}^\lambda\leq u_{*}^\lambda$ (see Marano \& Papageorgiou \cite[Proposition 2.1]{8}).

Suppose that $0<t<1$. Let $\rho=\|u_{*}^{\lambda}\|_{\infty}$ and let $\hat{\xi_{\rho}}>0$ be such that
$$x\mapsto c_{1}x^{q-1}-c_{4}(\lambda)x^{\tau-1}+\hat\xi_\rho x$$
is nondecreasing on $[0,\rho]$. We have
\begin{align*}
-\Delta(t v_{*}^\lambda)+(\xi(z)+\hat{\xi}_\rho)(t v_{*}^\lambda)= & t\left[-\Delta v_{*}^\lambda+(\xi(z)+\hat{\xi}_{\rho})v_{*}^\lambda\right]\\ \vspace{0.3cm}
= & t\left[c_1(v_{*}^\lambda)^{q-1}-c_{4}(\lambda)(v_{*}^\lambda)^{\tau-1}+\hat\xi_{\rho}v_{*}^\lambda\right]\\ \vspace{0.3cm}
< & c_1(t v_{*}^\lambda)^{q-1}-c_4(\lambda)(t v_{*}^\lambda)^{\tau-1}+\hat\xi_\rho(t v_{*}^\lambda)\\ \vspace{0.3cm}
&{}\hspace{2cm}(\mbox{since}\ t\in(0,1)\ \mbox{and}\ q<2<\tau)\\ \vspace{0.3cm}
\leq & c_1(u_{*}^\lambda)^{q-1}-c_{4}(u_{*}^\lambda)^{\tau-1}+\hat\xi_\rho u_{*}^\lambda\\ \vspace{0.3cm}
&{}\hspace{2cm}(\mbox{since}\ t v_{*}^\lambda\leq u_{*}^\lambda)\\ \vspace{0.3cm}
= & -\Delta u_{*}^\lambda+(\xi(z)+\hat{\xi}_{\rho})u_{*}^\lambda,
\end{align*}
\begin{align*}
&\Rightarrow \Delta (u_{*}^\lambda-t v_{*}^\lambda)<(\xi(z)+\hat\xi_{\rho})(u_{*}^\lambda-t v_{*}^\lambda)\leq(\|\xi^{+}\|_\infty+\hat\xi_{\rho})(u_{*}^\lambda-t v_{*}^\lambda)\\ \vspace{0.3cm}
&{}\hspace{7.3cm}(\mbox{see hypothesis}\ H(\xi)'),\\ \vspace{0.3cm}
&\Rightarrow u_{*}^\lambda-t v_{*}^\lambda\in\textrm{int}\, C_+\ (\mbox{by the strong maximum principle}).
\end{align*}
 However, this contradicts the maximality of $t\in(0,1)$. Hence $t\geq1$ and so we have $v_{*}^{\lambda}\leq u_{*}^\lambda$.

If in the above argument we reverse the roles of $u_{*}$ and $v_{*}$, we obtain
\begin{align*}
& u_{*}^\lambda\leq v_{*}^\lambda,\\ \vspace{0.3cm}
\Rightarrow\,  & u_{*}^\lambda=v_{*}^\lambda
\end{align*}
and this proves the uniqueness of the positive solution of problem ($A u_\lambda$).
\end{proof}
\vspace*{4pt}\noindent\textbf{Remark.} We can have an alternative proof of the uniqueness based on the Picone identity. The argument goes as follows. Suppose again that $v_{*}^\lambda\in H^{1}(\Omega)$ is another positive solution of ($A u_\lambda$). We have $v_{*}^\lambda\in\textrm{int}\, C_{+}$. Therefore
\begin{align}%\label{19}
& \int_{\Omega}\left(\frac{c_1}{(u_{*}^\lambda)^{2-q}}-c_4(\lambda)(u_{*}^\lambda)^{\tau-2}\right)\left((u_{*}^\lambda)^{2}-(v_{*}^\lambda)^{2}
\right)dz\nonumber\\ \vspace{0.3cm}
&\ =\int_\Omega
\left(c_1(u_{*}^\lambda)^{q-1}-c_4(\lambda)(u_{*}^\lambda)^{\tau-1}\right)\left(u_{*}^\lambda-\frac{(v_{*}^\lambda)^{2}}{u_{*}^\lambda}\right)dz
\nonumber\\ \vspace{0.3cm}
&\ =\int_\Omega\left(-\Delta_p u_{*}^\lambda+\xi(z)u_{*}^\lambda\right)\left(u_{*}^\lambda-\frac{(v_{*}^\lambda)^{2}}{u_{*}^\lambda}\right)dz
\nonumber\\ \vspace{0.3cm}
&\ =\int_\Omega\left(Du_{*}^\lambda,D\left(u_{*}^\lambda-\frac{(v_{*}^\lambda)^2}{u_{*}^\lambda}\right)\right)_{\RR^N}\,dz
+\int_\Omega\xi(z)u_{*}^\lambda\left(u_{*}^\lambda-\frac{(v_{*}^\lambda)^{2}}{u_{*}^\lambda}\right)dz\nonumber\\ \vspace{0.3cm}
&\ \ \ \ + \int_{\partial\Omega}\beta(z)u_{*}^{\lambda}\left(u_{*}^\lambda-\frac{(v_{*}^\lambda)^{2}}{u_{*}^\lambda}\right)d\sigma\nonumber\\ \vspace{0.3cm}
&\ \ \ \ (\mbox{using Green's identity, see \cite[p. 210]{6}})\nonumber%\\ \vspace{0.3cm}
\end{align}\begin{align}\label{19}& =\|Du_{*}^\lambda\|_{2}^{2}-\|Dv_{*}^\lambda\|_{2}^{2}+\int_\Omega R(v_{*}^\lambda,u_{*}^\lambda)dz+\int_\Omega\xi(z)((u_{*}^\lambda)^2-(v_{*}^\lambda)^2)dz\nonumber\\ \vspace{0.3cm}
&\ \ \ \ + \int_{\partial\Omega}\beta(z)((u_{*}^\lambda)^2-(v_{*}^\lambda)^2)d\sigma\\ \vspace{0.3cm}
&\ \ \ \ (\mbox{see the proof of Proposition}\ \ref{Proposition 5}).\nonumber
\end{align}
Interchanging the roles of $u_{*}^\lambda$ and $v_{*}^\lambda$ in the above argument, we obtain
\begin{align}\label{20}
& \int_{\Omega}\left(\frac{c_1}{(v_{*}^\lambda)^{2-q}}-c_4(\lambda)(v_{*}^\lambda)^{\tau-2}\right)\left((v_{*}^\lambda)^{2}-(u_{*}^\lambda)^{2}
\right)dz\nonumber\\ \vspace{0.3cm}
& =\|Dv_{*}^\lambda\|_{2}^{2}-\|Du_{*}^\lambda\|_{2}^{2}+\int_\Omega R(u_{*}^\lambda,v_{*}^\lambda)dz+\int_\Omega\xi(z)((v_{*}^\lambda)^2-(u_{*}^\lambda)^2)dz\nonumber\\ \vspace{0.3cm}
&\ \ \ \ + \int_{\partial\Omega}\beta(z)((v_{*}^\lambda)^2-(u_{*}^\lambda)^2)d\sigma.
\end{align}
We add \eqref{19} and \eqref{20} and use Picone's identity. We obtain
\begin{align}\label{21}
0 & \leq\int_\Omega\left[R(v_{*}^\lambda,u_{*}^\lambda)+R(u_{*}^\lambda,v_{*}^\lambda)\right]dz\nonumber \\ \vspace{0.3cm}
& =\int_\Omega\left[c_1\left(\frac{1}{(u_{*}^\lambda)^{2-q}}-\frac{1}{(v_{*}^\lambda)^{2-q}}\right)
-c_4(\lambda)\left((u_{*}^\lambda)^{\tau-2}-(v_{*}^\lambda)^{\tau-2}\right)\right]\left((u_{*}^\lambda)^{2}-(v_{*}^\lambda)^{2}\right)dz.
\end{align}
Since the function $x\mapsto\frac{c_1}{x^{2-q}}-c_4(\lambda)x^{\tau-2}$ is strictly decreasing on $(0,+\infty)$ (recall $q<2<\tau$),  it follows from \eqref{21} that
$$u_{*}^\lambda=v_{*}^\lambda.$$
So, we again get the uniqueness of the positive solution of $(Au_{\lambda})$.

Note also that since $\lambda\mapsto c_{4}(\lambda)$ is bounded on bounded sets of $\RR$, if $B\subseteq\RR$ is bounded and $\hat{c}_4\geq c_4(\bar\lambda)$ for all $\bar\lambda\in B$, then the unique solution $\bar{u}\in\textrm{int}\, C_+$ of
$$\left\{\begin{array}{ll}
-\Delta u(z)+\xi(z)u(z)=c_1u(z)^{q-1}-\hat{c}_4u(z)^{\tau-1}\ \ \mbox{in}\ \Omega,\\ [0.3cm]
\displaystyle\frac{\partial u}{\partial n}+\beta(z)u=0\ \ \mbox{on}\ \partial\Omega,
\end{array}\right\}
$$
satisfies $\bar{u}\leq u_{*}^{\bar{\lambda}}$ for all $\bar{\lambda}\in B$.

Using Proposition \ref{Proposition 7}, we can produce a lower bound for the solution set $S(\lambda)$.
\begin{proposition}\label{Proposition 8}
If hypotheses $H(\xi)'$, $H(\beta)$, $H_1$ hold and $\lambda\in\mathcal{L}=(-\infty,\widehat{\lambda}_1)$, then $u_{*}^{\lambda}\leq u$ for all $u\in S(\lambda)$.
\end{proposition}

\begin{proof}
Let $u\in S(\lambda)$ and consider the following Carath\'{e}odory function
\begin{eqnarray}\label{22}\hat{g}_{\lambda}(z,x)=\left\{\begin{array}{ll}
0 & \mbox{if}\ x<0\\ [0.3cm]
c_1 x^{q-1}-c_4(\lambda)x^{\tau-1}+\mu x & \mbox{if}\ 0\leq x\leq u(z)\\ [0.3cm]
c_1u(z)^{q-1}-c_4(\lambda)u(z)^{\tau-1}+\mu u(z) & \mbox{if}\ u(z)<x.
\end{array}\right.
\end{eqnarray}
Here $\mu>0$ is as in \eqref{6}. We set $\widehat{G}_\lambda(z,x)=\int_{0}^{x}\hat{g}_\lambda(z,s)ds$ and consider the $C^1$-functional $\widehat\psi_\lambda:H^1(\Omega)\rightarrow\RR$ defined by
$$\widehat\psi_\lambda(u)=\frac{1}{2}\gamma(u)+\frac{\mu}{2}\|u\|_2^2-\int_\Omega\widehat{G}_\lambda(z,u)dz \ \mbox{for all}\ u\in H^1(\Omega).$$

 Invoking \eqref{6} and \eqref{22}, we see that $\widehat\psi_{\lambda}$ is coercive. It is also sequentially weakly lower semicontinuous. Therefore we can find $\tilde{u}_{*}^{\lambda}\in H^1(\Omega)$ such that
\begin{eqnarray}\label{23}
\widehat\psi_\lambda(\tilde u_{*}^\lambda)=\inf\left[\psi_\lambda(u):\, u\in H^1(\Omega)\right].
\end{eqnarray}

As before (see the proof of Proposition \ref{Proposition 6}), exploiting the fact that $q<2<\tau$, we can show that
\begin{align*}
 & \widehat\psi(\tilde u_{*}^\lambda)<0=\widehat\psi_\lambda(0), \\ \vspace{0.3cm}
\Rightarrow\, & \tilde u_{*}^\lambda\neq0.
\end{align*}
By \eqref{23} we have
\begin{align}\label{24}
 & \widehat\psi_\lambda'(\tilde u_{*}^\lambda)=0,\nonumber \\ \vspace{0.3cm}
\Rightarrow\, & \langle A(\tilde u_{*}^\lambda),h\rangle+\int_\Omega\xi(z)\tilde u_{*}^\lambda h\, dz+\int_{\partial\Omega}\beta(z)\tilde u_{*}^\lambda h\, d\sigma+\mu\int_{\Omega}\tilde u_{*}^\lambda h\, dz=\int_{\Omega}\hat g_\lambda(z,\tilde u_{*}^\lambda)h\, dz \\ \vspace{0.3cm}
& \hspace{8cm}\ \mbox{for all}\ h\in H^{1}(\Omega).\nonumber
\end{align}
In \eqref{24} we choose $h=-(\tilde u_{*}^\lambda)^-\in H^1(\Omega)$ and obtain
\begin{align*}
 & \gamma((\tilde u_{*}^\lambda)^-)+\mu\|(\tilde u_{*}^\lambda)^-\|_2^2=0\ (\mbox{see \eqref{22}}) \\ \vspace{0.3cm}
\Rightarrow\, & c_1\|(\tilde u_{*}^\lambda)^-\|^2\leq0\ (\mbox{see \eqref{6}}) \\ \vspace{0.3cm}
\Rightarrow\, & \tilde u_{*}^\lambda\geq0,\ \tilde u_{*}^\lambda\neq0.
\end{align*}
Next, in \eqref{24} we choose $(\tilde u_{*}^\lambda-u)^+\in H^1(\Omega)$. Then
\begin{align*}
\langle & A(\tilde u_{*}^\lambda),(\tilde u_{*}^\lambda-u)^+\rangle+\int_\Omega\limits\xi(z)\tilde u_{*}^\lambda(\tilde u_{*}^\lambda-u)^+ dz+\int_{\partial\Omega}\limits\beta(z)\tilde u_{*}^\lambda(\tilde u_{*}^\lambda-u)^+d\sigma\\
&\qquad+\int_\Omega\limits\mu\tilde u_{*}^\lambda(\tilde u_{*}^\lambda-u)^+dz \\
& =\int_\Omega\limits\left(c_1 u^{q-1}-c_4(\lambda)u^{\tau-1}\right)(\tilde{u}_{*}^\lambda-u)^{+}dz+\mu\int_\Omega\limits u(\tilde u_{*}^\lambda-u)^+dz\ (\mbox{see \eqref{22}}) \\
& \leq\int_\Omega\limits(\lambda u+f(z,u))(\tilde u_{*}^\lambda-u)^{+}dz+\mu\int_\Omega\limits u(\tilde u_{*}^\lambda-u)^+dz\ (\mbox{see \eqref{14}}) \\
& =\langle A(u),(\tilde u_{*}^\lambda-u)^+\rangle+\int_\Omega\limits\xi(z) u(\tilde u_{*}^\lambda-u)^+dz+\int_{\partial\Omega}\limits\beta(z)u(\tilde u_{*}^\lambda-u)^+d\sigma\\
&+\int_{\Omega}\limits\mu u(\tilde u_{*}^\lambda-u)^+dz,
 \hspace{2cm} (\mbox{since}\ u\in S(\lambda))
\end{align*}
\begin{align*}
&\Rightarrow\, \gamma((\tilde u_{*}^\lambda-u)^+)+\mu\|(\tilde u_{*}^\lambda-u)^+\|_2^2\leq0,\\ \vspace{0.3cm}
&\Rightarrow\, c_0\|(\tilde u_{*}^\lambda-u)^+\|^2\leq0\ (\mbox{see \eqref{6}}),\\ \vspace{0.3cm}
&\Rightarrow\, \tilde u_{*}^\lambda\leq u.
\end{align*}
Therefore we have proved that
\begin{eqnarray}\label{25}
\tilde u_{*}^\lambda\in[0,u]=\{v\in H^1(\Omega):\, 0\leq v(z)\leq u(z)\, \mbox{for a.a}\ z\in \Omega\},\; \tilde u_{*}^\lambda\neq0.
\end{eqnarray}
By \eqref{22} and \eqref{25}, we see that equation \eqref{24} becomes
\begin{align*}
&\langle A(\tilde u_{*}^\lambda),h\rangle+\int_\Omega\xi(z)\tilde u_{*}^\lambda h\, dz+\int_{\partial\Omega}\beta(z)\tilde u_{*}^\lambda h\, d\sigma=\int_\Omega\left[c_1(\tilde u_{*}^\lambda)^{q-1}-c_4(\lambda)(u_{*}^\lambda)^{\tau-1}\right]dz \\ \vspace{0.3cm}
& \hspace{8cm} \mbox{for all}\ h\in H^1(\Omega), \\ \vspace{0.3cm}
&\Rightarrow\, \tilde u_{*}^\lambda\ \mbox{is a positive solution of}\ (Au_\lambda), \\ \vspace{0.3cm}
&\Rightarrow\, \tilde u_{*}^\lambda=u_{*}^\lambda\in\textrm{int}\, C_{+}\ (\mbox{see Proposition \ref{Proposition 7}}).
\end{align*}
Therefore $u_{*}^\lambda\leq u$ for all $u\in S(\lambda)$.
\end{proof}

This lower bound leads to the existence of a smallest positive solution for problem ($P_\lambda$), $\lambda\in\mathcal{L}=(-\infty,\widehat\lambda_1)$.

\begin{proposition}\label{Proposition 9}
If hypotheses $H(\xi)'$, $H(\beta)$, $H_{1}$ hold and $\lambda\in\mathcal{L}=(-\infty,\widehat{\lambda}_1)$, then problem $(P_\lambda)$ admits a smallest positive solution $\bar u_\lambda\in S(\lambda)\subseteq\textrm{int}\, C_{+}$.
\end{proposition}

\begin{proof}
As in Filippakis \& Papageorgiou \cite{5}, we can show that $S(\lambda)$ is downward directed, that is, if $u_1,u_2\in S(\lambda)$, then we can find $u\in S(\lambda)$ such that  $u\leq u_1$, $u\leq u_2$.
For completeness we sketch a proof. So, given $\varepsilon>0$ consider the function
$$\vartheta_\varepsilon (s)=\left\{\begin{array}{ll}
-\varepsilon\quad&\mbox{if}\ s<-\varepsilon\\
s\quad&\mbox{if}\ s\in[-\varepsilon,\varepsilon]\\
\varepsilon\quad&\mbox{if}\ s>\varepsilon\,.
\end{array}\right.
$$
Evidently, $\vartheta_\varepsilon(\cdot)$ is Lipschitz and so $\vartheta_\varepsilon ((u_1-u_2)^-)\in H^1(\Omega)$. Moreover, the chain rule for Sobolev functions implies that
$$D\vartheta_\varepsilon ((u_1-u_2)^-)=\vartheta_\varepsilon'((u_1-u_2)^-)D(u_1-u_2)^-.$$
Let $\psi\in C^1(\overline\Omega))$. Then we introduce the test functions
$$\eta_1=\vartheta_\varepsilon ((u_1-u_2)^-)\psi\quad\mbox{and}\quad \eta_2=(\varepsilon-\vartheta_\varepsilon ((y_1-y_2)^-)\psi),$$
which belong to $H^1(\Omega)\cap L^\infty (\Omega)$. We have
$$\langle \gamma'(u_1),\eta_1\rangle=\lambda\int_\Omega u_1\eta_1dz+\int_\Omega f(z,u_1)\eta_1dz,$$
$$\langle \gamma'(u_2),\eta_2\rangle=\lambda\int_\Omega u_2\eta_2dz+\int_\Omega f(z,u_2)\eta_2dz.$$
We add these two equalities and divide by $\varepsilon>0$. Taking into account that
$$\frac 1\varepsilon\,\vartheta_\varepsilon ((u_1-u_2)^-)(z)\rightarrow\chi_{\{u_1<u_2\}}(z)\quad \mbox{for a.a.}\ z\in\Omega\ \mbox{as}\ \varepsilon\rightarrow 0^+$$
and
$$\chi_{\{u_1\geq u_2\}}=1-\chi_{\{u_1<u_2\}}$$
we obtain
{\small$$\begin{array}{ll}
\hspace*{-10pt}&\displaystyle \langle \gamma'(u_1),\chi_{\{u_1<u_2\}}\psi\rangle+
\langle \gamma'(u_2),\chi_{\{u_1\geq u_2\}}\psi\rangle=\\
\hspace*{-15pt}&\displaystyle \lambda\int_{\{u_1<u_2\}}u_1\psi dz+ \lambda\int_{\{u_1\geq u_2\}}u_2\psi dz+
\int_{\{u_1<u_2\}}f(z,u_1)\psi dz+\int_{\{u_1\geq u_2\}}f(z,u_2)\psi dz.
\end{array}$$}

So, if $\bar{u}=\min\{u_1,u_2\}$, then $\bar{u}$ is an upper solution of ($P_\lambda$) and so by standard truncation techniques we can find $u\in S(\lambda)$ such that
$0\leq u\leq \bar{u}\leq\left\{\begin{array}{l} u_1\\ u_2\end{array}\right.$.

\smallskip
Then Lemma 3.10 of Hu \& Papageorgiou \cite[p. 178]{7}, implies that there exist $u_n\in S(\lambda)$, $n\in\NN$, $\{u_n\}_{n\geq1}$ decreasing such that
$$\inf S(\lambda)=\inf_{n\geq1}u_n.$$
For every $n\in\NN$, we have
\begin{align}\label{26}
&\langle A(u_n),h\rangle+\int_\Omega\xi(z)u_n h\, dz+\int_{\partial\Omega}\beta(z)u_n h\,d\sigma=\int_{\Omega}[\lambda u_n+f(z,u_n)]h\,dz \\ \vspace{0.3cm}
& \hspace{8cm}\ \mbox{for all}\ h\in H^1(\Omega).\nonumber
\end{align}
Evidently, $\{u_n\}_{n\geq1}\subseteq H^1(\Omega)$ is bounded and so we may assume that
\begin{eqnarray}\label{27}
u_n\xrightarrow{w}\bar u_\lambda\ \mbox{in}\ H^1(\Omega)\ \mbox{and}\ u_n\rightarrow\bar u_\lambda\ \mbox{in}\ L^2(\Omega)\ \mbox{and in}\ L^2(\partial\Omega).
\end{eqnarray}
In \eqref{26} we pass to the limit as $n\rightarrow\infty$ and use \eqref{27}. Then
\begin{align*}
&\langle A(\bar u_\lambda),h\rangle+\int_\Omega\xi(z)\bar u_\lambda h\, dz+\int_{\partial\Omega}\beta(z)\bar u_\lambda h\,d\sigma=\int_{\Omega}[\lambda \bar u_\lambda+f(z,\bar u_\lambda)]h\, dz \\ \vspace{0.3cm}
& \hspace{8cm}\ \mbox{for all}\ h\in H^1(\Omega), \\ \vspace{0.3cm}
\Rightarrow\, & \bar u_\lambda\ \mbox{is a solution of}\ (P_\lambda).
\end{align*}
Due to Proposition \ref{Proposition 8} we know that
\begin{align*}
&u_{*}^\lambda\leq u_n\ \mbox{for all}\ n\in\NN, \\ \vspace{0.3cm}
\Rightarrow\, & u_{*}^\lambda\leq\bar u_\lambda\ (\mbox{see \eqref{27}}),\\ \vspace{0.3cm}
\Rightarrow\, & \bar u_\lambda\in S(\lambda)\subseteq\textrm{int}\, C_{+}\ \mbox{and}\ \bar u_\lambda=\inf{S(\lambda)}.
\end{align*}
\end{proof}

We examine the monotonicity and continuity properties of the map $\lambda\mapsto\bar u_\lambda$ from $\mathcal{L}=(-\infty,\widehat\lambda_1)$ into $C^1(\overline\Omega)$.

\begin{proposition}\label{Proposition 10}
If hypotheses $H(\xi)'$, $H(\beta)$, $H_1$ hold, then the map $\lambda\mapsto\bar{u}_\lambda$ is nondecreasing and left continuous from $\mathcal{L}=(-\infty,\widehat{\lambda}_1)$ into $C^1(\overline\Omega)$.
\end{proposition}

\begin{proof}
First, we establish the monotonicity of the map $\lambda\rightarrow\bar u_\lambda$. So let $\lambda<\eta<\widehat{\lambda}_1$ and consider $\bar u_\eta\in S(\eta)\subseteq\textrm{int}\, C_+$ the minimal positive solution of problem ($P_{\eta}$). We introduce the following Carath\'{e}odory function
\begin{eqnarray}\label{28}e_{\lambda}(z,x)=\left\{\begin{array}{ll}
0 & \mbox{if}\ x<0\\ [0.3cm]
(\lambda+\mu)x+f(z,x) & \mbox{if}\ 0\leq x\leq u_\eta(z)\\ [0.3cm]
(\lambda+\mu)u_\eta(z)+f(z,u_\eta(z)) & \mbox{if}\ u_\eta(z)<x.
\end{array}\right.
\end{eqnarray}
As always, $\mu>0$ is as in \eqref{6}. We set $E_\lambda(z,x)=\int_{0}^{x}e_\lambda(z,s)ds$ and consider the $C^1$-functional $w_{\lambda}:H^1(\Omega)\rightarrow\RR$ defined by
$$w_\lambda(u)=\frac{1}{2}\gamma(u)+\frac{\mu}{2}\|u\|_2^2-\int_\Omega E_\lambda(z,u)dz\ \mbox{for all}\ u\in H^1(\Omega).$$

 It follows from \eqref{6} and \eqref{28} that $w_\lambda(\cdot)$ is coercive. It is also sequentially weakly lower semicontinuous. So, we can find $\tilde u_{\lambda}\in H^{1}(\Omega)$ such that
\begin{align}\label{29}
& w_\lambda(\tilde u_\lambda)=\inf\left[w_\lambda(u):\, u\in H^1(\Omega)\right]<0=w_\lambda(0)\\ \vspace{0.3cm}
& \hspace{4cm}(\mbox{as before since}\ q<2).\nonumber
\end{align}
By \eqref{29} we have $\tilde u_\lambda\neq0$ and
\begin{align}\label{30}
&w_\lambda'(\tilde u_\lambda)=0,\nonumber \\ \vspace{0.3cm}
\Rightarrow & \langle A(\tilde u_\lambda),h\rangle+\int_\Omega\xi(z)\tilde u_\lambda h\,dz+\int_{\partial\Omega}\beta(z)\tilde u_\lambda h\, d\sigma+\mu\int_{\Omega}\tilde u_\lambda h\, dz=\int_{\Omega}e_{\lambda}(z,\tilde u_\lambda)h\, dz\\ \vspace{0.3cm}
& \hspace{8cm}\ \mbox{for all}\ h\in H^1(\Omega).\nonumber
\end{align}
As in the proof of Proposition \ref{Proposition 8}, choosing in \eqref{30}, first $h=-\tilde u_\lambda^{-}\in H^{1}(\Omega)$ and then $(\tilde u_{\lambda}-\bar u_\eta)^+\in H^1(\Omega)$, we can show that
\begin{eqnarray}\label{31}
\tilde u_\lambda\in[0,\bar u_\eta]=\left\{v\in H^1(\Omega):\, 0\leq v(z)\leq\bar u_\eta(z)\ \mbox{for a.a.}\ z\in\Omega\right\},\ \tilde u_\lambda\neq0.
\end{eqnarray}
Then we can infer from \eqref{28}, \eqref{30}, \eqref{31} that
\begin{align*}
& \tilde u_\lambda\in S(\lambda),\\ \vspace{0.3cm}
\Rightarrow\, & \bar u_\lambda\leq\tilde u_\lambda\leq\bar u_\eta,\\ \vspace{0.3cm}
\Rightarrow\, & \lambda\mapsto\bar u_\lambda\ \mbox{from}\ \mathcal{L}=(-\infty,\widehat\lambda_1)\ \mbox{into}\ C^{1}(\overline\Omega)\ \mbox{is nondecreasing}.
\end{align*}

Next, we establish the left continuity of this map. To this end, let $\lambda_n\rightarrow\lambda^-$ with $\lambda\in\mathcal{L}=(-\infty,\widehat\lambda_1)$. Evidently, $\{\bar u_{\lambda_n}\}_{n\geq1}\subseteq H^1(\Omega)$ is bounded and increasing. Therefore we may assume that
\begin{eqnarray}\label{32}
\bar u_{\lambda_n}\xrightarrow{w}\hat u_\lambda\ \mbox{in}\ H^1(\Omega)\ \mbox{and}\ \bar u_{\lambda_n}\rightarrow\hat u_\lambda\ \mbox{in}\ L^2(\Omega)\ \mbox{and in}\ L^2(\partial\Omega).
\end{eqnarray}
We have
\begin{align}\label{33}
& \langle A(\bar u_{\lambda_n}),h\rangle+\int_{\Omega}\xi(z)\bar u_{\lambda_n}h\, dz+\int_{\partial\Omega}\beta(z)\bar u_{\lambda_n}h\, d\sigma=
\int_\Omega\left[\lambda_n\bar u_{\lambda_n}+f(z,\bar u_{\lambda_n})\right]h\, dz \\ \vspace{0.3cm}
& \hspace{8cm}\ \mbox{for all}\ h\in H^1(\Omega),\ \mbox{all}\ n\in\NN.\nonumber
\end{align}
In \eqref{33} we pass to the limit as $n\rightarrow\infty$ and use \eqref{32}. We obtain
\begin{align}\label{34}
& \langle A(\hat u_{\lambda}),h\rangle+\int_{\Omega}\xi(z)\hat u_{\lambda}h\, dz+\int_{\partial\Omega}\beta(z)\hat u_{\lambda}h\, d\sigma=
\int_\Omega\left[\lambda\hat u_{\lambda}+f(z,\hat u_{\lambda})\right]h\, dz \\ \vspace{0.3cm}
& \hspace{9cm}\ \mbox{for all}\ h\in H^1(\Omega).\nonumber
\end{align}

Set $B=\{\lambda_n\}_{n\geq1}$ and let $\hat{c}_4\geq c_4(\bar\lambda)$ for all $\bar\lambda\in B$ (recall that $\lambda\mapsto c_4(\lambda)$ is bounded on bounded sets of $\RR$). Let $\bar u\in\textrm{int}\, C_+$ be the unique positive solution of the following semilinear Robin problem
\begin{align*}&\left\{\begin{array}{ll}
-\Delta u(z)+\xi(z)u(z)=c_1 u(z)^{q-1}-\hat c_4 u(z)^{\tau-1}\ \ \mbox{in}\ \Omega,\\ [0.3cm]
\displaystyle\frac{\partial u}{\partial n}+\beta(z)u=0\ \ \mbox{on}\ \partial\Omega.
\end{array}\right\}\\
&\hspace{6cm}(\mbox{see Proposition \ref{Proposition 7}}).
\end{align*}
We know that $\bar u\leq\bar u_{\lambda_n}\ \mbox{for all}\ n\in\NN$ (see the remark following Proposition \ref{Proposition 7}). Hence
\begin{align*}
& \bar u\leq\hat{u}_\lambda,\\ \vspace{0.3cm}
\Rightarrow\, & \hat u_{\lambda}\in S(\lambda)\subseteq\textrm{int}\, C_{+}.
\end{align*}
We claim that $\hat u_\lambda=\bar u_\lambda$. If this is not true, then we can find $z_{0}\in\overline\Omega$ such that
\begin{eqnarray}\label{35}
\bar{u}_\lambda(z_0)<\hat u_\lambda(z_0).
\end{eqnarray}
From Wang \cite{14}, we know that we can find $M_1>0$ and $\alpha\in(0,1)$ such that
\begin{eqnarray}\label{36}
\bar{u}_{\lambda_n}\in C^{1,\alpha}(\overline\Omega)\ \mbox{and}\ \|\bar u_{\lambda_n}\|_{C^1(\overline\Omega)}\leq M_1\ \mbox{for all}\ n\in\NN.
\end{eqnarray}
Exploiting the compact embedding of $C^{1,\alpha}(\overline\Omega)$ into $C^1(\overline\Omega)$ and using \eqref{32}, we obtain
\begin{align}\label{37}
& \bar u_{\lambda_n}\rightarrow\hat u_\lambda\ \mbox{in}\ C^1(\overline\Omega),\\ \vspace{0.3cm}
\Rightarrow\, & \bar u_{\lambda}(z_0)<\bar u_{\lambda_n}(z_{0})\ \mbox{for all}\, n\geq n_{0}\nonumber
\end{align}
a contradiction to the monotonicity of $\lambda\mapsto\bar u_\lambda$ (recall that $\lambda_n<\lambda$ for all $n\in\NN$). So $\hat{u}_\lambda=\bar u_{\lambda}$ and this proves the left continuity of $\lambda\mapsto\bar{u}_\lambda$ from $\mathcal{L}=(-\infty,\widehat\lambda_1)$ into $C^1(\overline\Omega)$ (see \eqref{37}).
\end{proof}

We can improve the monotonicity of the map $\lambda\mapsto\bar u_\lambda$ provided that we strengthen the conditions on $f(z,\cdot)$.
\begin{enumerate}
\item [$H_{2}$:] $f:\Omega\times\RR\rightarrow\RR$ is a Carath\'{e}odory function such that  $f(z,0)=0$ for a.a. $z\in\Omega$, hypotheses $H_2$(i), (ii), (iii) are the same as hypotheses $H_1$(i), (ii), (iii) and%\\
\begin{enumerate}
\item [(iv)] for every $\rho>0$, there exists $\hat\xi_\rho>0$ such that  for a.a. $z\in\Omega$, the function
$$x\mapsto f(z,x)+\hat\xi_\rho x$$
is nondecreasing on $[0,\rho]$.
\end{enumerate}
\end{enumerate}

\vspace*{4pt}\noindent\textbf{Remark.} The new condition on $f(z,\cdot)$ is satisfied, if for a.a. $z\in\Omega$, $f(z,\cdot)$ is differentiable and $z\mapsto f_x'(z,\cdot)$ is locally $L^{\infty}(\Omega)$-bounded (just use the mean value theorem).%\\ [0.3cm]

\vspace*{4pt}\noindent\textbf{Examples.} The following functions satisfy hypotheses $H_2$. As before, for the sake of simplicity, we drop the $z$-dependence
$$f_{1}(x)=x^{q-1}\ \mbox{for all}\ x\geq0\ \mbox{with}\ 1<q<2,$$
$$
f_{2}(x)=\left\{\begin{array}{ll}
x^{q-1} & \mbox{if}\ x\in[0,1]\\ [0.3cm]
2x^{\tau-1}-x^{s-1} & \mbox{if}\ 1<x
\end{array}\right. \mbox{with}\ 1<q,\tau,s<2,\; s<\tau,
$$
$$
f_{3}(x)=\left\{\begin{array}{ll}
x^{q-1}-x^{s-1} & \mbox{if}\ x\in[0,1]\\ [0.3cm]
x^{\tau-1}\ln x & \mbox{if}\ 1<x
\end{array}\right. \mbox{with}\ 1<q,s,\tau<2,\; q<s,
$$
$$
f_{4}(x)=\left\{\begin{array}{ll}
\ln(x^{q-1}+1) & \mbox{if}\ x\in[0,1]\\ [0.3cm]
x^{\tau-1}-x^{s-1}+c & \mbox{if}\ 1<x
\end{array}\right. \mbox{with}\ 1<q,\tau,s<2,\; s<\tau,\; c=\ln2.
$$

\begin{proposition}\label{Proposition 11}
If hypotheses $H(\xi)'$, $H(\beta)$, $H_2$ hold, then the map $\lambda\mapsto\bar u_\lambda$ from $\mathcal{L}=(-\infty,\widehat\lambda_1)$ into $C^1(\overline\Omega)$ is strictly increasing (that is, if $\lambda<\eta<\widehat\lambda_1$, then $\bar u_\eta-\bar u_\lambda\in\textrm{int}\, C_+$).
\end{proposition}

\begin{proof}
Let $\lambda<\eta<\widehat\lambda_1$ and let $\bar u_\lambda\in S(\lambda)\subseteq\textrm{int}\, C_+$, $\bar u_\eta\in S(\eta)\subseteq\textrm{int}\, C_+$ be the corresponding minimal positive solutions of problems ($P_\lambda$) and ($P_\eta$), respectively. By Proposition \ref{Proposition 10} we know that
\begin{eqnarray}\label{38}
\bar u_\lambda\leq\bar u_\eta.
\end{eqnarray}
Let $\rho=\|\bar u_\eta\|_\infty$ and let $\hat\xi_\rho>0$ be as postulated by hypothesis $H_2$(iv). We have
\begin{align*}
& -\Delta\bar u_\lambda(z)+(\xi(z)+\hat{\xi}_\rho)\bar u_\lambda(z)\\ \vspace{0.3cm}
& {}\hspace{0.3cm} = \lambda\bar u_\lambda(z)+f(z,\bar u_\lambda(z))+\hat\xi_\rho\bar u_\lambda(z)\\ \vspace{0.3cm}
& {}\hspace{0.3cm} <\eta\bar u_\eta(z)+f(z,\bar u_\eta(z))+\hat\xi_\rho\bar u_\eta(z)\\ \vspace{0.3cm}
& \hspace{1cm}(\mbox{see \eqref{38}, hypothesis}\ H_2\mbox{(iv) and recall}\ \lambda<\eta)\\ \vspace{0.3cm}
& {}\hspace{0.3cm} =-\Delta\bar u_\eta(z)+(\xi(z)+\hat\xi_\rho)\bar u_\eta(z)\ \mbox{for a.a.}\ z\in\Omega,\\ \vspace{0.3cm}
\Rightarrow\, & \Delta(\bar u_\eta-\bar u_\lambda)(z)\leq(\|\xi^+\|_\infty+\hat\xi_\rho)(\bar u_\eta-\bar u_\lambda)(z)\ \mbox{for a.a.}\ z\in\Omega\\ \vspace{0.3cm}
& \hspace{1cm}(\mbox{see hypothesis}\ H(\xi)'),\\ \vspace{0.3cm}
\Rightarrow\, & \bar u_\eta-\bar u_\lambda\in\textrm{int}\, C_+\\ \vspace{0.3cm}
& \hspace{1cm}(\mbox{by the strong maximum principle, see \cite[p. 738]{6}}).
\end{align*}
This proves the strict monotonicity of $\lambda\mapsto\bar u_\lambda$.
\end{proof}

In fact, under a monotonicity restriction on the quotient $\frac{f(z,x)}{x}$, we can conclude that for all $\lambda\in\mathcal{L}=(-\infty,\widehat\lambda_1)$ problem ($P_\lambda$) admits a unique positive solution.

Hence the new conditions on the perturbation $f(z,x)$ are the following:
\begin{enumerate}
\item [$H_{3}$:] $f:\Omega\times\RR\rightarrow\RR$ is a Carath\'{e}odory function such that  $f(z,0)=0$ for a.a. $z\in\Omega$, hypotheses $H_3$(i), (ii), (iii) are the same as hypotheses $H_1$(i), (ii), (iii) and%\\
\begin{enumerate}
\item [(iv)] for a.a. $z\in\Omega$, $x\mapsto\frac{f(z,x)}{x}$ is strictly decreasing on $(0,+\infty)$.
\end{enumerate}
\end{enumerate}

\vspace*{4pt}\noindent
\textbf{Examples.} The following functions satisfy the new conditions. Again, for the sake of simplicity we drop the $z$-dependence:
$$f_{1}(x)=x^{q-1}\ \mbox{for all}\ x\geq0\ \mbox{with}\ 1<q<2,$$
$$
f_{2}(x)=\left\{\begin{array}{ll}
x^{q-1} & \mbox{if}\ x\in[0,1]\\ [0.3cm]
x^{\tau-1} & \mbox{if}\ 1<x
\end{array}\right. \mbox{with}\ 1<q,\, \tau<q.
$$

\begin{proposition}\label{Proposition 12}
If hypotheses $H(\xi)'$, $H(\beta)$, $H_3$ hold and $\lambda\in\mathcal{L}=(-\infty,\widehat\lambda_1)$, then $S(\lambda)$ is a singleton, that is, $S(\lambda)=\{\bar u_\lambda\}$ and $\lambda\mapsto\bar u_\lambda$ is nondecreasing and continuous from $(-\infty,\widehat{\lambda}_1)$ into $C^{1}(\overline\Omega)$.
\end{proposition}

\begin{proof}
The nonemptiness of $S(\lambda)$ follows from Proposition \ref{Proposition 6}. Suppose that $\bar u_{\lambda},\bar v_\lambda\in S(\lambda)\subseteq\textrm{int}\, C_{+}$. We have
\begin{eqnarray}\label{39}
-\Delta\bar u_\lambda(z)+\xi(z)\bar{u}_\lambda(z)=\lambda\bar u_\lambda(z)+f(z,\bar u_\lambda(z))\ \mbox{for a.a.}\ z\in\Omega,
\end{eqnarray}
\begin{eqnarray}\label{40}
-\Delta\bar v_\lambda(z)+\xi(z)\bar v_\lambda(z)=\lambda\bar v_\lambda(z)+f(z,\bar v_\lambda(z))\ \mbox{for a.a.}\ z\in\Omega.
\end{eqnarray}
We multiply \eqref{39} with $\bar v_\lambda(z)$ and \eqref{40} with $\bar u_\lambda(z)$, then integrate both equations over $\Omega$ and use Green's identity. We obtain
\begin{align}\label{41}
&\int_\Omega(D\bar u_\lambda,D\bar v_\lambda)_{\RR^N}dz+\int_{\Omega}\limits\xi(z)\bar u_\lambda\bar v_\lambda dz+\int_{\partial\Omega}\limits\beta(z)\bar u_\lambda\bar v_\lambda d\sigma\nonumber\\=&\lambda\int_\Omega\limits\bar u_\lambda\bar v_\lambda dz+\int_\Omega \limits f(z,\bar u_\lambda)\bar v_\lambda dz
\end{align}
\begin{align}\label{42}
&\int_\Omega\limits(D\bar v_\lambda,D\bar u_\lambda)_{\RR^N}dz+\int_{\Omega}\limits\xi(z)\bar v_\lambda\bar u_\lambda dz+\int_{\partial\Omega}\limits\beta(z)\bar v_\lambda\bar u_\lambda d\sigma\nonumber\\=&\lambda\int_\Omega\limits\bar v_\lambda\bar u_\lambda dz+\int_\Omega \limits f(z,\bar v_\lambda)\bar u_\lambda dz.
\end{align}
We subtract \eqref{42} from \eqref{41} and obtain
\begin{align*}
& \int_\Omega\left[f(z,\bar u_\lambda)\bar v_\lambda-f(z,\bar v_\lambda)\bar u_\lambda\right]dz=0,\\ \vspace{0.3cm}
\Rightarrow\, & \int_{\Omega}\left[\frac{f(z,\bar u_\lambda)}{\bar u_\lambda}-\frac{f(z,\bar v_\lambda)}{\bar v_\lambda}\right]\bar u_\lambda\bar v_\lambda dz=0\ (\mbox{recall}\ \bar u_\lambda,\bar v_\lambda\in\textrm{int}\, C_{+})\\ \vspace{0.3cm}
\Rightarrow\, & \bar u_\lambda=\bar v_\lambda\ (\mbox{see hypothesis}\ H_3\mbox{(iv)}).
\end{align*}
This proves the uniqueness of the positive solution of problem ($P_\lambda$). The uniqueness of this positive solution, together with Proposition \ref{Proposition 10}, imply that the map $\lambda\mapsto\bar u_\lambda$ from $\mathcal{L}=(-\infty,\widehat\lambda_1)$ into $C^1(\overline\Omega)$ is nondecreasing and continuous.
\end{proof}

As before, by strengthening the conditions on the perturbation $f(z,\cdot)$ we can improve the monotonicity of the map $\lambda\rightarrow\bar u_\lambda$.

The new condition on $f(z,x)$ are the following:
\begin{enumerate}
\item [$H_{4}$:] $f:\Omega\times\RR\rightarrow\RR$ is a Carath\'{e}odory function such that  $f(z,0)=0$ for a.a. $z\in\Omega$, hypotheses $H_4$(i), (ii), (iii), (iv) are the same as the corresponding hypotheses $H_3$(i), (ii), (iii), (iv) and%\\
\begin{enumerate}
\item [(v)] for every $\rho>0$, there exists $\hat\xi_\rho>0$ such that  for a.a. $z\in\Omega$ the function $x\mapsto f(z,x)+\hat\xi_\rho x$ is nondecreasing on $[0,\rho]$.
\end{enumerate}
\end{enumerate}

\vspace*{4pt}\noindent
\textbf{Remark.} The examples after hypotheses $H_3$, also satisfy hypotheses $H_4$. Then Propositions \ref{Proposition 11} and \ref{Proposition 12} imply the following result.

\begin{proposition}\label{Proposition 13}
If hypotheses $H(\xi)'$, $H(\beta)$, $H_4$ hold and $\lambda\in\mathcal{L}=(-\infty,\widehat\lambda_1)$, then the set $S(\lambda)$ is a singleton $\{\bar u_\lambda\}$ and the map $\lambda\mapsto \bar u_\lambda$ from $\mathcal{L}$ into $C^1(\overline\Omega)$ is strictly increasing (that is, if $\lambda<\eta<\widehat\lambda_1$, then $\bar u_\eta-\bar u_\lambda\in\textrm{int}\, C_+$).
\end{proposition}
Summarizing the situation for problem ($P_\lambda$) when the perturbation term $f(z,\cdot)$ is sublinear, we can state the following theorem.
\begin{theorem}\label{Theorem 14}
\begin{enumerate}
\item [(a)] If hypotheses $H(\xi)'$, $H(\beta)$, $H_1$ hold, then for every $\lambda\geq\widehat\lambda_1$ we have $S(\lambda)=\emptyset$, while for every $\lambda<\widehat\lambda_1$, $S(\lambda)\neq\emptyset$, $S(\lambda)\subseteq\textrm{int}\, C_+$, problem ($P_\lambda$) admits a smallest positive solution $\bar u_\lambda\in\textrm{int}\, C_+$ and the map $\lambda\mapsto\bar u_\lambda$ from $\mathcal{L}=(-\infty,\widehat\lambda_1)$ into $C^1(\overline\Omega)$ is nondecreasing (that is, $\lambda<\eta<\lambda_1$ implies $\bar u_\lambda\leq\bar u_\eta$) and left continuous.
\item [(b)] If hypotheses $H(\xi)'$, $H(\beta)$, $H_2$ hold, then the map $\lambda\mapsto\bar u_\lambda$ from $\mathcal{L}=(-\infty,\widehat\lambda_1)$ into $C^1(\overline\Omega)$ is strictly increasing (that is, $\lambda<\eta<\widehat\lambda_1$ implies $\bar u_\eta-\bar u_\lambda\in\textrm{int}\, C_{+}$).
\item [(c)] If hypotheses $H(\xi)'$, $H(\beta)$, $H_3$ hold and $\lambda\in\mathcal{L}=(-\infty,\widehat\lambda_1)$, then $S(\lambda)$ is a singleton $\{\bar u_\lambda\}$ ($\bar u_\lambda\in\textrm{int}\, C_+$) and the map $\lambda\mapsto\bar u_\lambda$ from $\mathcal{L}=(-\infty,\widehat\lambda_1)$ into $C^1(\overline\Omega)$ is nondecreasing and continuous.
\item [(d)] If hypotheses $H(\xi)'$, $H(\beta)$, $H_4$ hold, then the map $\lambda\mapsto\bar u_\lambda$ is strictly increasing from $\mathcal{L}=(-\infty,\widehat\lambda_1)$ into $C^1(\overline\Omega)$.
\end{enumerate}
\end{theorem}

\section{The superlinear case}\label{sec4}

In this section, we investigate the case of a superlinear perturbation $f(z,\cdot)$. Now we cannot have uniqueness of the positive solution and in fact we show that the problem exhibits a kind of bifurcation phenomenon, that is, for $\lambda<\widehat\lambda_1$, problem ($P_\lambda$) has at least two positive solutions, while  for $\lambda\geq\widehat\lambda_1$, $S(\lambda)=\emptyset$. We stress that for the superlinearity of $f(z,\cdot)$ we do not use the Ambrosetti-Rabinowitz condition (AR-condition for short).

The hypotheses on the perturbation $f(z,x)$ are the following:
\begin{enumerate}
\item [$H_{5}$:] $f:\Omega\times\RR\rightarrow\RR$ is a Carath\'{e}odory function such that  for a.a. $z\in\Omega$, $f(z,0)=0$, $f(z,x)\geq0$ for all $x\geq0$, $f(z,x)>0$ for a.a. $z\in\Omega_0\subseteq\Omega$ with $|\Omega_{0}|_{N}>0$, all $x>0$ and%\\
\begin{enumerate}
\item [(i)] $f(z,x)\leq a(z)(1+x^{r-1})$ for a.a. $z\in\Omega$, all $x\geq0$, with $a\in L^{\infty}(\Omega)_{+}$ and $2<r<2^{*}$;
\item [(ii)] if $F(z,x)=\int_0^x f(z,s)ds$, then
$$\lim_{x\rightarrow+\infty}\frac{F(z,x)}{x^2}=+\infty\ \mbox{uniformly for a.a.}\ z\in\Omega$$
and there exists $q\in\left(\max\left\{1,(r-2)\frac{N}{2}\right\},2^{*}\right)$ such that
$$0<\tilde\xi\leq\liminf_{x\rightarrow+\infty}\frac{f(z,x)x-2F(z,x)}{x^q}\ \mbox{uniformly for a.a.}\ z\in\Omega;$$
\item [(iii)] $\displaystyle\lim_{x\rightarrow0^{+}}\limits\frac{f(z,x)}{x}=0$ uniformly for a.a. $z\in\Omega$.
\end{enumerate}
\end{enumerate}

\vspace*{4pt}\noindent
\textbf{Remarks.} As in the \textit{sublinear} case, since we are looking for positive solutions and the above hypotheses concern the positive semiaxis $\RR_+=[0,+\infty)$,  we may assume without any loss of generality that $f(z,x)=0$ for a.a. $z\in\Omega$, all $x\leq0$. Hypothesis $H_{5}$(ii) implies that for a.a. $z\in\Omega$, $f(z,\cdot)$ is superlinear near $+\infty$. However, we do not assume the usual for superlinear problems AR-condition. Recall that the AR-condition (unilateral version since it is imposed only on the positive semiaxis), says that there exist $\tau>2$ and $M_2>0$ such that
\begin{eqnarray}\label{43}
0<\tau F(z,x)\leq f(z,x)x\ \mbox{for a.a.}\ z\in\Omega,\ \mbox{all}\ x\geq M_2
\end{eqnarray}
\begin{eqnarray}\label{44}
0<\essinf_{\Omega}F(\cdot,M_2)
\end{eqnarray}
(see Ambrosetti \& Rabinowitz \cite{2} and Mugnai \cite{10}). Integrating \eqref{43} and using \eqref{44}, we obtain
\begin{eqnarray}\label{45}
c_8 x^\tau\leq F(z,x)\ \mbox{for a.a.}\ z\in\Omega,\ \mbox{all}\ x\geq M_2,\ \mbox{some}\ c_{8}>0.
\end{eqnarray}
From \eqref{43} and \eqref{45} it follows that for a.a. $z\in\Omega$, $f(z,\cdot)$ has at least $(\tau-1)$-polynomial growth near $+\infty$. Our hypothesis is implied by the AR-condition. We may assume that $\tau>\max\left\{1,(r-2)\frac{N}{2}\right\}$. We have
\begin{align*}
\frac{f(z,x)x-2F(z,x)}{x^\tau} & =\frac{f(z,x)x-\tau F(z,x)}{x^\tau}+(\tau-2)\frac{F(z,x)}{x^\tau}\\ \vspace{0.3cm}
& \geq c_8(\tau-2)\ \mbox{for a.a.}\ z\in\Omega,\ \mbox{all}\ x\geq M_2\\ \vspace{0.3cm}
& \hspace{3cm}(\mbox{see \eqref{43} and \eqref{45}}).
\end{align*}
 Thus, hypothesis $H_5$(ii) is satisfied (recall that $\tau>2$). This more general superlinearity condition, incorporates in our framework superlinear nonlinearities with \textit{slower} growth near $+\infty$, which fail to satisfy the AR-condition (unilateral version).

\vspace*{4pt}\noindent\textbf{Examples.} The following functions satisfy hypotheses $H_5$. As in the previous examples, for the sake of simplicity we drop the $z$-dependence.
$$f_{1}(x)=x^{q-1}\ \mbox{for all}\ x\geq0,\ \mbox{with}\ 2<q<2^{*}$$
$$f_{2}(x)=x \ln(1+x)\ \mbox{for all}\ x\geq0.$$
Note that $f_1$ satisfies the unilateral AR-condition, but $f_2$ does not.

Using Proposition \ref{Proposition 4} and reasoning exactly as in the proof of Proposition \ref{Proposition 5}, we obtain the following result.

\begin{proposition}\label{Proposition 15}
If hypotheses $H(\xi)'$, $H(\beta)$, $H_{5}$ hold and $\lambda\geq\lambda_{1}$, then $S(\lambda)=\emptyset$.
\end{proposition}
So according to this proposition, we have $\mathcal{L}\subseteq(-\infty,\widehat\lambda_1)$. In fact we will show that equality holds.

\begin{proposition}\label{Proposition 16}
If hypotheses $H(\xi)'$, $H(\beta)$, $H_{5}$ hold, then $\mathcal{L}=(-\infty,\widehat\lambda_1)$.
\end{proposition}

\begin{proof}
We fix $\lambda\in(-\infty,\widehat\lambda_1)$ and consider the Carath\'{e}odory function
\begin{eqnarray}\label{46}
k_{\lambda}(z,x)=\left\{\begin{array}{ll}
0 & \mbox{if}\ x\leq0\\ [0.3cm]
\lambda x+f(z,x) & \mbox{if}\ 0<x.
\end{array}\right.
\end{eqnarray}
We set $K_\lambda(z,x)=\int_0^x k_\lambda(z,s)ds$ and consider the $C^1$-functional $\varphi_\lambda:H^1(\Omega)\rightarrow\RR$ defined by
$$\varphi_\lambda(u)=\frac{1}{2}\gamma(u)+\frac{\mu}{2}\|u^-\|_2^2-\int_\Omega K_\lambda(z,u)dz\ \mbox{for all}\ u\in H^1(\Omega).$$
Here, $\mu>0$ is as in \eqref{6}. Hypotheses $H_5$(i), (iii) imply that given $\varepsilon>0$, we can find $c_9=c_9(\varepsilon)>0$ such that
\begin{eqnarray}\label{47}
F(z,x)\leq\frac{\varepsilon}{2}x^2+c_9x^r\ \mbox{for a.a.}\ z\in\Omega,\ \mbox{all}\ x\geq0.
\end{eqnarray}
Choosing $\varepsilon\in(0,\widehat\lambda_1-\lambda)$ (recall $\lambda<\widehat\lambda_1$), for all $u\in H^1(\Omega)$ we obtain
\begin{align}\label{48}
\varphi_\lambda(u) & \geq\frac{1}{2}\gamma(u^-)+\frac{\mu}{2}\|u^-\|_2^2+\frac{1}{2}\gamma(u^+)-\frac{\lambda+\varepsilon}{2}\|u^+\|_2^2
-c_9\|u^+\|_r^r\nonumber \\ \vspace{0.3cm}
& \hspace{3cm}(\mbox{see \eqref{46}, \eqref{47}})\nonumber \\ \vspace{0.3cm}
& \geq \frac{c_0}{2}\|u^-\|^2+\frac{\hat c}{2}\|u^+\|^2-c_{10}\|u\|^r\ \mbox{for some}\ c_{10}>0\nonumber \\ \vspace{0.3cm}
& \hspace{3cm}(\mbox{see \eqref{6} and Lemma \ref{Lemma 3}})\nonumber \\ \vspace{0.3cm}
& \geq c_{11}\|u\|^2-c_{10}\|u\|^r\ \mbox{for some}\ c_{11}>0.
\end{align}
Since $r>2$, from \eqref{48} we infer that $u=0$ is a strict local minimizer of $\varphi_\lambda$. It is easy to see that $K_{\varphi_\lambda}\subseteq C_{+}$ (see \eqref{46}). Thus, we may assume that $u=0$ is an isolated critical point of $\varphi_\lambda$ or otherwise we already have a whole sequence of distinct positive solutions of ($P_\lambda$) which converge to zero in $H^1(\Omega)$. Therefore we can find $\rho\in(0,1)$ small such that
\begin{align}\label{49}
& \varphi_\lambda(0)=0<\inf\left[\varphi_\lambda(u):\, \|u\|=\rho\right]=m_\rho^\lambda\\
& (\mbox{see Aizicovici, Papageorgiou \& Staicu \cite{1}, proof of Proposition 29}).\nonumber
\end{align}
Hypothesis $H_5$(ii) implies that given any $u\in\textrm{int}\, C_{+}$, we have
\begin{eqnarray}\label{50}
\varphi_\lambda(tu)\rightarrow-\infty\ \mbox{as}\ t\rightarrow+\infty.
\end{eqnarray}%\\

\vspace*{4pt}\noindent\textbf{Claim.} \emph{$\varphi_\lambda$ satisfies the C-condition.}\vspace*{4pt}

Let $\{u_n\}_{n\geq1}\subseteq H^1(\Omega)$ such that
\begin{eqnarray}\label{51}
|\varphi_\lambda(u_n)|\leq M_3\ \mbox{for some}\ M_3>0,\ \mbox{all}\ n\in\NN,
\end{eqnarray}
\begin{eqnarray}\label{52}
(1+\|u_n\|)\varphi_\lambda'(u_n)\rightarrow0\ \mbox{in}\ H^1(\Omega)^{*}\ \mbox{as}\ n\rightarrow\infty.
\end{eqnarray}
By \eqref{52} we have
\begin{align}\label{53}
& \left|\langle A(u_n),h\rangle+\int_\Omega\limits\xi(z)u_n h\, dz+\int_{\partial\Omega}\limits\beta(z)u_n h\, d\sigma+\mu\int_\Omega\limits u_n^- h\, dz-\int_\Omega\limits k_\lambda(z,u_n)h\, dz\right|\\ \vspace{0.3cm}
& \hspace{5cm}\leq\frac{\varepsilon_n\|h\|}{1+\|u_n\|}\qquad \mbox{for all}\ h\in H^1(\Omega)\ \mbox{with}\ \varepsilon_n\rightarrow0^+.\nonumber
\end{align}
In \eqref{53} we choose $h=-u_n^-\in H^1(\Omega)$. Using \eqref{46}, we obtain
\begin{align}\label{54}
& \left|\gamma(u_n^-)+\mu\|u_n^-\|_2^2\right|\leq \varepsilon_n\ \mbox{for all}\ n\in\NN,\nonumber \\ \vspace{0.3cm}
\Rightarrow\, & c_0\|u_n^-\|^2\leq\varepsilon_n\ \mbox{for all}\ n\in\NN\ (\mbox{see \eqref{6}})\nonumber \\ \vspace{0.3cm}
\Rightarrow\, & u_n^-\rightarrow0\ \mbox{in}\ H^1(\Omega).
\end{align}
It follows from \eqref{51} and \eqref{54} that
\begin{align}\label{55}
& \gamma(u_n^+)-\int_\Omega\left[\lambda(u_n^+)^2+2F(z,u_n^+)\right]dz\leq M_4\\ \vspace{0.3cm}
& \hspace{3cm} \mbox{for some}\ M_4>0,\ \mbox{all}\ n\in\NN. \nonumber
\end{align}
On the other hand, if in \eqref{53} we choose $h=u_n^+\in H^1(\Omega)$, then
\begin{align}\label{56}
& -\gamma(u_n^+)+\int_\Omega\left[\lambda(u_n^+)^2+f(z,u_n^+)u_n^+\right]dz\leq\varepsilon_n\ \mbox{for all}\ n\in\NN\\ \vspace{0.3cm}
& \hspace{7cm} (\mbox{see \eqref{46}}).\nonumber
\end{align}
Adding \eqref{55} and \eqref{56}, we obtain
\begin{eqnarray}\label{57}
\int_\Omega\left[f(z,u_n^+)u_n^+-2F(z,u_n^+)\right]dz\leq M_5\ \mbox{for some}\ M_5>0,\ \mbox{all}\ n\in\NN.
\end{eqnarray}
Hypotheses $H_5$(i), (ii) imply that we can find $\tilde\xi_0\in(0,\xi_0)$ and $c_{12}>0$ such that
\begin{eqnarray}\label{58}
\tilde\xi_0 x^q-c_{12}\leq f(z,x)x-2 F(z,x)\ \mbox{for a.a.}\ z\in\Omega,\ \mbox{all}\ x\geq0.
\end{eqnarray}

We use \eqref{58} in \eqref{57} and infer that
\begin{eqnarray}\label{59}
\{u_n^+\}_{n\geq1}\subseteq L^q(\Omega)\ \mbox{is bounded}.
\end{eqnarray}
First, suppose that $N\neq2$. By hypothesis $H_5$(ii), we see that  we may assume
without loss of generality
that $q<r<2^{*}$. So, we can find $t\in(0,1)$ such that
\begin{eqnarray}\label{60}
\frac{1}{r}=\frac{1-t}{q}+\frac{t}{2^{*}}.\end{eqnarray}
Invoking the interpolation inequality (see, for example, Gasinski \& Papageorgiou \cite[p. 905]{6}), we have
\begin{align}\label{61}
& \|u_n^+\|_r \leq\|u_n^+\|_q^{1-t}\|u_{n}^{+}\|_{2^{*}}^t\nonumber \\ \vspace{0.3cm}
& \hspace{1cm} \leq M_6\|u_n^+\|^t\ \mbox{for some}\ M_6>0,\ \mbox{all}\ n\in\NN \ (\mbox{see \eqref{59}}),\nonumber \\ \vspace{0.3cm}
\Rightarrow\, & \|u_n^+\|_r^r\leq M_7\|u_n^+\|^{tr}\ \mbox{for some}\ M_7=M_6^r>0,\ \mbox{all}\ n\in\NN.
\end{align}
In \eqref{53} we choose $h=u_n^+\in H^1(\Omega)$. Then
\begin{align}\label{62}
& \gamma(u_n^+)\leq\varepsilon_n+\int_\Omega\left[\lambda(u_n^+)^2+f(z,u_n^+)u_n^+\right]dz\ \mbox{for all}\ n\in\NN\ (\mbox{see \eqref{46}}),\nonumber \\ \vspace{0.3cm}
\Rightarrow\, & \gamma(u_n^+)\leq c_{13}(1+\|u_n^+\|_r^r)\ \mbox{for some}\ c_{13}>0,\ \mbox{all}\ n\in\NN\nonumber \\ \vspace{0.3cm}
& \hspace{2cm}(\mbox{see hypothesis}\ H_5\mbox{(i) and recall that}\ 2<r) \nonumber \\ \vspace{0.3cm}
& \hspace{1cm}\leq c_{14}(1+\|u_n^+\|^{tr})\ \mbox{for some}\ c_{14}>0, \mbox{all}\ n\in\NN \\ \vspace{0.3cm}
& \hspace{2cm}(\mbox{see \eqref{61}}).\nonumber
\end{align}
 By hypothesis $H_5$(i), we see that we can always assume that $r$ is close to $2^{*}$, hence $q\geq2$ (see hypothesis $H_5$(ii)). Then \eqref{59} implies that $\{u_n^+\}_{n\geq1}\subseteq L^2(\Omega)$ is bounded and so by \eqref{62} we have
\begin{align}\label{63}
& \gamma(u_n^+)+\mu\|u_n^+\|_2^2\leq c_{15}(1+\|u_n^+\|^{tr})\ \mbox{for some}\ c_{15}>0,\ \mbox{all}\ n\in\NN,\nonumber \\ \vspace{0.3cm}
\Rightarrow\, & c_0\|u_n^+\|^2\leq c_{15}(1+\|u_n^+\|^{tr})\ \mbox{for all}\ n\in\NN\\ \vspace{0.3cm}
& \hspace{6cm}(\mbox{see \eqref{6}}).\nonumber
\end{align}
Due to \eqref{60} and hypothesis $H_5$(ii), we see that
\begin{align}\label{64}
& t\,r<2,\nonumber \\ \vspace{0.3cm}
\Rightarrow\, & \{u_n^+\}_{n\geq1}\subseteq H^1(\Omega)\ \mbox{is bounded (see \eqref{63})}.
\end{align}
If $N=2$, then $2^{*}=+\infty$ and the Sobolev embedding theorem says that $H^1(\Omega)\hookrightarrow L^\eta(\Omega)$ for all $\eta\in[1,+\infty)$. Let $\eta>r>q$ and $t\in(0,1)$ such that
\begin{align*}
& \frac{1}{r}=\frac{1-t}{q}+\frac{t}{\eta},\\ \vspace{0.3cm}
\Rightarrow\, & t\,r=\frac{\eta(r-q)}{\eta-q}.
\end{align*}
Note that
$$\frac{\eta(r-q)}{\eta-q}\rightarrow r-q\ \mbox{as}\ \eta\rightarrow+\infty=2^{*}\ \mbox{and}\ r-q<2\ (\mbox{see hypothesis}\ H_5\mbox{(ii)}).$$
Therefore the previous argument works if instead of $2^{*}$ we use $\eta>1$ big such that  $t\,r<2$. We again obtain \eqref{64}. It follows from \eqref{54} and \eqref{64} that
$$\{u_n\}_{n\geq1}\subseteq H^1(\Omega)\ \mbox{is bounded}.$$
Thus, we may assume that
\begin{eqnarray}\label{65}
u_n\xrightarrow{w} u\ \mbox{in}\ H^1(\Omega)\ \mbox{and}\ u_n\rightarrow u\ \mbox{in}\ L^r(\Omega)\ \mbox{and in}\ L^2(\partial\Omega).
\end{eqnarray}
In \eqref{53} we choose $h=u_n-u\in H^1(\Omega)$, pass to the limit as $n\rightarrow\infty$ and use \eqref{65}. Then
\begin{align*}
& \lim_{n\rightarrow\infty}\langle A(u_n),u_n-u\rangle=0,\\ \vspace{0.3cm}
\Rightarrow\, & \|Du_n\|_2\rightarrow\|Du\|_2,%\\ \vspace{0.3cm}
\end{align*}\begin{align*}\Rightarrow\, & Du_n\rightarrow Du\ \mbox{in}\ L^2(\Omega,\RR^N)\\ \vspace{0.3cm}
& \vspace{5cm}\ (\mbox{by the Kadec-Klee property, see \eqref{65}})\\ \vspace{0.3cm}
\Rightarrow\, & u_n\rightarrow u\ \mbox{in}\ H^1(\Omega)\ (\mbox{see \eqref{65}}).
\end{align*}
Therefore $\varphi_\lambda$ satisfies the C-condition and this proves the Claim.

Then \eqref{49}, \eqref{50} and the Claim, permit the use of Theorem \ref{Theorem 1} (the mountain pass theorem) and so we can find $u_\lambda\in H^1(\Omega)$ such that
\begin{align*}
& u_{\lambda}\in K_{\varphi_\lambda}\ \mbox{and}\ m_\rho^\lambda\leq\varphi_\lambda(u_\lambda),\\ \vspace{0.3cm}
\Rightarrow\, & u_\lambda\neq0\ (\mbox{see \eqref{49}) and so}\ u_\lambda\in S(\lambda)\subseteq\textrm{int}\, C_{+}.
\end{align*}
\end{proof}

Next, we show that for every $\lambda\in(-\infty,\widehat\lambda_1)$, problem ($P_\lambda$) has a smallest positive solution.

First, let us recall the following lemma from Hu and Papageorgiou \cite[p. 178]{7}.

\vspace*{4pt}\noindent
{\bf Lemma A.} {\em If $(\Omega, \sum, \mu)$ is a finite measure space and ${\mathcal D}$ is a family of $\RR_+$-valued measurable functions which is downward directed, then there exists a unique (modulo equality $\mu$-a.e.) function $h:\Omega\rightarrow \RR_+$ such that\\
(a) $h(\omega)\leq u(\omega)$ $\mu$-a.e. in $\Omega$ for all $u\in {\mathcal D}$;\\
(b) if $g:\Omega\rightarrow \RR$ is a measurable function such that
$$g(\omega)\leq u(\omega)\ \ \mbox{$\mu$-a.e. in $\Omega$ for all $u\in {\mathcal D}$}$$
then $g(\omega)\leq h(\omega)$ $\mu$-a.e.,
that is, $h=\inf {\mathcal D}$. Moreover, there is a decreasing sequence $\{u_n\}_{n\geq 1}\subseteq{\mathcal D}$ such that $\inf_{n\geq 1}u_n=h=\inf{\mathcal D}$.}

\begin{proposition}\label{Proposition 17}
If hypotheses $H(\xi)'$, $H(\beta)$, $H_5$ hold and $\lambda\in\mathcal{L}=(-\infty,\widehat\lambda_1)$, then problem ($P_\lambda$) admits a smallest positive solution $\bar u_\lambda\in\textrm{int}\, C_{+}$.
\end{proposition}

\begin{proof}
We again have that $S(\lambda)$ is downward directed (that is, if $u_1,u_2\in S(\lambda)$, then we can find $u\in S(\lambda)$ such that  $u\leq u_1$, $u\leq u_2$; see Filippakis \& Papageorgiou \cite{5}). Therefore Lemma A above implies that we can find a decreasing sequence $\{u_n\}_{n\geq1}\subseteq S(\lambda)$ such that
$$\inf S(\lambda)=\inf_{n\geq1}u_n.$$
We have
\begin{align}\label{66}
& \langle A(u_n),h\rangle+\int_\Omega\xi(z)u_n h\, dz+\int_{\partial\Omega}\beta(z)u_n h\,d\sigma=\int_{\Omega}[\lambda u_n+f(z,u_n)]hdz\\ \vspace{0.3cm}
& \hspace{7cm}\ \mbox{for all}\ h\in H^1(\Omega),\ \mbox{all}\ n\in\NN.\nonumber
\end{align}
Choosing $h=u_n\in H^1(\Omega)$\ \mbox{in \eqref{66}}, we obtain
\begin{eqnarray}\label{67}
\gamma(u_n)=\lambda\|u_n\|_2^2+\int_\Omega f(z,u_n)u_n dz\ \mbox{for all}\ n\in\NN.
\end{eqnarray}
Since $u_{n}\leq u_1\in\textrm{int}\, C_+$ for all $n\in\NN$ (recall that $\{u_n\}_{n\geq1}\subseteq H^1(\Omega)$ is decreasing) and by \eqref{67}, \eqref{6} and hypotheses $H(\beta)'$, $H(\beta)$, $H_5$(i), we can conclude that
$$\{u_n\}_{n\geq1}\subseteq H^1(\Omega)\ \mbox{is bounded}.$$
Hence, we may assume that
\begin{eqnarray}\label{68}
u_n\xrightarrow{w}\bar u_\lambda\ \mbox{in}\ H^1(\Omega)\ \mbox{and}\ u_n\rightarrow\bar u_\lambda\ \mbox{in}\ L^2(\Omega)\ \mbox{and in}\ L^2(\partial\Omega).
\end{eqnarray}

Returning to \eqref{66}, passing to the limit as $n\rightarrow\infty$ and using \eqref{68}, we obtain
\begin{align*}
& \langle A(\bar u_\lambda),h\rangle+\int_\Omega\xi(z)\bar u_\lambda h\, dz+\int_{\partial\Omega}\beta(z)\bar u_\lambda h\,d\sigma=\int_{\Omega}[\lambda\bar u_\lambda+f(z,\bar u_\lambda)]h\, dz\\ \vspace{0.3cm}
& \hspace{8cm}\ \mbox{for all}\ h\in H^1(\Omega),\\ \vspace{0.3cm}
\Rightarrow\, & \bar u_\lambda\geq0\ \mbox{is a solution of}\ (P_\lambda).
\end{align*}

We will show that $\bar u_\lambda\neq0$. Arguing by contradiction, suppose that $\bar u_\lambda=0$. Let $y_n=\frac{u_n}{\|u_n\|}$, $n\in\NN$. Then $\|y_n\|=1$, $y_n\in\textrm{int}\, C_+$ for all $n\in\NN$. So we may assume that
\begin{eqnarray}\label{69}
y_n\xrightarrow{w}y\ \mbox{in}\ H^1(\Omega)\ \mbox{and}\ u_n\rightarrow y\ \mbox{in}\ L^r(\Omega)\ \mbox{and}\ L^{2}(\partial\Omega).
\end{eqnarray}
By \eqref{66} we have
\begin{align}\label{70}
& \langle A(y_n),h\rangle+\int_\Omega\xi(z)y_n h\, dz+\int_{\partial\Omega}\beta(z)y_n h\, d\sigma=\int_\Omega\left[\lambda y_n+\frac{N_f(u_n)}{\|u_n\|}\right]h\, dz\\ \vspace{0.3cm}
& \hspace{8cm}\ \mbox{for all}\ h\in H^1(\Omega),\ \mbox{all}\ n\in\NN.\nonumber
\end{align}
Let $\rho=\|u_1\|_\infty$. Hypotheses $H_5$(i), (iii) imply that
\begin{align}\label{71}
& f(z,x)\leq c_{16}\, x\ \mbox{for a.a.}\ z\in\Omega,\ \mbox{all}\ x\in[0,\rho],\ \mbox{some}\ c_{16}>0,\nonumber\\ \vspace{0.3cm}
\Rightarrow\, & \frac{f(z,u_n(z))}{\|u_n\|}\leq c_{18}\, y_n(z)\ \mbox{for a.a.}\ z\in\Omega,\ \mbox{all}\ n\in\NN,\nonumber\\ \vspace{0.3cm}
\Rightarrow\, & \left\{\frac{N_f(u_n)}{\|u_n\|}\right\}_{n\geq1}\subseteq L^2(\Omega)\ \mbox{is bounded (see \eqref{69})}.
\end{align}
By passing to a subsequence if necessary and using hypothesis $H_5$(iii) we infer that
\begin{eqnarray}\label{72}
\frac{N_f(u_n)}{\|u_n\|}\xrightarrow{w}0\ \mbox{in}\ L^2(\Omega)
\end{eqnarray}
(see Aizicovici, Papageorgiou \& Staicu \cite{1}, proof of Proposition 14). In \eqref{70}, we first choose $h=y_n-y\in H^1(\Omega)$, pass to the limit as $n\rightarrow\infty$ and use \eqref{69} and \eqref{71}. Then
\begin{align}
& \lim_{n\rightarrow\infty}\langle A(y_n),y_n-y\rangle=0,\nonumber\\ \vspace{0.3cm}
\label{73} \Rightarrow\, & y_n\rightarrow y\ \mbox{in}\ H^1(\Omega)\\
& \hspace{2cm}\ (\mbox{by the Kadec-Klee property})\nonumber\\ \vspace{0.3cm}
\label{74} \Rightarrow\, & \|y\|=1.
\end{align}

Next in \eqref{70} we choose $h=y_n\in H^1(\Omega)$ and pass to the limit as $n\rightarrow\infty$. Using \eqref{72} and \eqref{73} we obtain
$$\gamma(y)=\lambda\|y\|_{2}^{2}<\widehat\lambda_1\|y\|_2^2\ (\mbox{see \eqref{74}}),$$
a contradiction to \eqref{2}. Hence, $\bar u_\lambda\neq0$ and therefore
$$\bar u_\lambda\in S(\lambda)\subseteq\textrm{int}\, C_+,\ \bar u_\lambda=\inf S(\lambda).$$
\end{proof}

Reasoning as in the proof of Proposition \ref{Proposition 10}, we can establish the monotonicity and continuity properties of the map $\lambda\mapsto\bar u_\lambda$.

\begin{proposition}\label{Proposition 18}
If hypotheses $H(\xi)'$, $H(\beta)$, $H_5$ hold, then the map $\lambda\mapsto\bar u_\lambda$ from $\mathcal{L}=(-\infty,\widehat\lambda_1)$ into $C^1(\overline\Omega)$ is nondecreasing and continuous.
\end{proposition}

As before (see Proposition \ref{Proposition 11} and hypotheses $H_2$), by strengthening the conditions on $f(z,\cdot)$, we can improve the monotonicity of $\lambda\mapsto\bar u_\lambda$.

The new conditions on $f(z,x)$ are the following:
\begin{enumerate}
\item [$H_{6}$:] $f:\Omega\times\RR\rightarrow\RR$ is a Carath\'{e}odory function such that  $f(z,0)=0$ for a.a. $z\in\Omega$, hypotheses $H_6$(i), (ii), (iii) are the same as the corresponding hypotheses $H_5$(i), (ii), (iii) and
    \begin{enumerate}
    \item [(iv)] for every $\rho>0$, there exists $\hat\xi_\rho>0$ such that  for a.a. $z\in\Omega$, the function
    $$x\mapsto f(z,x)+\hat\xi_\rho x$$
    is nondecreasing on $[0,\rho]$.
    \end{enumerate}
\end{enumerate}\vspace{0.3cm}

Reasoning as in the proof of Proposition \ref{Proposition 9}, we obtain.

\begin{proposition}\label{Proposition 19}
If hypotheses $H(\xi)'$, $H(\beta)$, $H_6$ hold, then the map $\lambda\mapsto\bar u_\lambda$ from $\mathcal{L}=(-\infty,\widehat\lambda_1)$ into $C^1(\overline\Omega)$ is strictly increasing (that is, if $\lambda<\eta<\widehat\lambda_1$, then $\bar u_\eta-\bar u_\lambda\in\textrm{int}\, C_+$).
\end{proposition}

Next, we show that for all admissible $\lambda\in\mathcal{L}=(-\infty,\widehat\lambda_1)$, problem ($P_\lambda$) has at least two positive solutions, which are ordered.

\begin{proposition}\label{Proposition 20}
If hypotheses $H(\xi)'$, $H(\beta)$, $H_6$ hold and $\lambda\in\mathcal{L}=(-\infty,\widehat\lambda_1)$, then problem ($P_\lambda$) admits at least two positive solutions
$$u_\lambda,\hat u_\lambda\in\textrm{int}\, C_+,\ \hat{u}_\lambda-u_\lambda\in\textrm{int}\, C_+.$$
\end{proposition}

\begin{proof}
 By Proposition \ref{Proposition 16} we already have one positive solution $u_\lambda\in\textrm{int}\, C_+$. We may assume that $u_\lambda$ is the minimal positive solution (that is, $u_\lambda=\bar u_\lambda$, see Proposition \ref{Proposition 17}). We consider the following Carath\'{e}odory function
\begin{eqnarray}\label{75}
\tilde g_{\lambda}(z,x)=\left\{\begin{array}{ll}
(\lambda+\mu)u_\lambda(z)+f(z,u_\lambda(z)) & \mbox{if}\ x\leq u_\lambda(z)\\ [0.3cm]
(\lambda+\mu)x+f(z,x) & \mbox{if}\ u_\lambda(z)<x.
\end{array}\right.
\end{eqnarray}
We set $\widetilde G_\lambda(z,x)=\int_0^x\tilde g_\lambda(z,s)ds$ and consider the $C^1$-functional $\widetilde\psi_\lambda:H^1(\Omega)\rightarrow\RR$ defined by
$$\widetilde\psi_\lambda(u)=\frac{1}{2}\gamma(u)+\frac{\mu}{2}\|u\|_2^2-\int_\Omega\widetilde G_\lambda(z,u)dz\ \mbox{for all}\ u\in H^{1}(\Omega).$$
From \eqref{75}, we see that on $[u_\lambda )=\{u\in H^1(\Omega):\ u_\lambda(z)\leq u(z)\ \mbox{for a.a.}\ z\in\Omega\}$ we have
\begin{align}\label{76}
& \widetilde\psi_\lambda=\varphi_\lambda+\tilde\xi_\lambda\ \mbox{for some}\ \widetilde\xi_\lambda\in\RR,\nonumber \\ \vspace{0.3cm}
\Rightarrow\, & \widetilde\psi_\lambda\ \mbox{satisfies the C-condition}\\ \vspace{0.3cm}
& \hspace{3cm}\ (\mbox{see the Claim in the proof of Proposition \ref{Proposition 16}}).\nonumber
\end{align}

\vspace*{4pt}\noindent\textbf{Claim.} \emph{We may assume that $u_\lambda\in\textrm{int}\, C_+$ is a local minimizer of the functional $\widetilde\psi_\lambda$.}

Let $\lambda<\eta<\widehat\lambda_1$ and let $u_\eta\in S(\eta)$. From Proposition \ref{Proposition 19}, we obtain $u_\eta-u_\lambda\in\textrm{int}\, C_+$. We introduce the following truncation of $\tilde g_\lambda(z,\cdot)$:
\begin{eqnarray}\label{77}
g_{\lambda}^{*}(z,x)=\left\{\begin{array}{ll}
\widetilde g_{\lambda}(z,x) & \mbox{if}\ x\leq u_\eta(z)\\ [0.3cm]
\widetilde g_{\lambda}(z,u_\eta(z)) & \mbox{if}\ u_\eta(z)<x.
\end{array}\right.
\end{eqnarray}
Evidently, this is a Carath\'{e}odory function. We set $G_{\lambda}^{*}(z,x)=\int_0^x g_\lambda^{*}(z,s)ds$ and consider the $C^1$-functional $\psi_\lambda^{*}: H^1(\Omega)\rightarrow\RR$ defined by
$$\psi_\lambda^{*}(u)=\frac{1}{2}\gamma(u)+\frac{\mu}{2}\|u\|_2^2-\int_\Omega G_\lambda^{*}(z,u)dz\ \mbox{for all}\ u\in H^1(\Omega).$$
As before (see the proof of Proposition \ref{Proposition 10}), we can show that
\begin{eqnarray}\label{78}
K_{\psi_\lambda^{*}}\subseteq[u_\lambda,u_\eta]=\left\{u\in H^1(\Omega):\, u_\lambda(z)\leq u(z)\leq u_\eta(z)\ \mbox{for a.a.}\ z\in\Omega\right\}.
\end{eqnarray}
Moreover, by \eqref{6} and \eqref{77}, we see that $\psi_\lambda^{*}$ is coercive. It is also  sequentially weakly lower semicontinuous. So we can find $u_\lambda^{*}\in H^1(\Omega)$ such that
\begin{align*}
& \psi_\lambda^{*}(u_\lambda^{*})=\inf\left[\psi_\lambda^{*}(u):\, u\in H^1(\Omega)\right],\\ \vspace{0.3cm}
\Rightarrow\, & u_\lambda^{*}\in K_{\psi_\lambda^{*}}\subseteq[u_\lambda,u_\eta]\ (\mbox{see \eqref{78}}),\\ \vspace{0.3cm}
\Rightarrow\, & u_\lambda^{*}\in S(\lambda)\subseteq\textrm{int}\, C_{+}\ (\mbox{see \eqref{77}}).
\end{align*}

If $u_\lambda^{*}\neq u_\lambda$, then this is the second positive solution of problem ($P_\lambda$) and, as we will see in the last part of the proof, we have $u_\lambda^{*}-u_\lambda\in\textrm{int}\, C_{+}$, so we are done. So, suppose $u_\lambda^{*}=u_\lambda$. We have
$$\psi_\lambda^{*}|_{[0,u_\eta]}=\widetilde\psi_\lambda|_{[0,u_\eta]}\ (\mbox{see \eqref{75} and \eqref{77}}).$$
Since $u_\lambda=u_\lambda^{*}$ and $u_\eta-u_\lambda\in\textrm{int}\, C_+$ (see Proposition \ref{Proposition 19}), it follows that
\begin{align*}
&u_\lambda\ \mbox{is a local}\ C^1(\overline\Omega)\mbox{-minimizer of}\ \widetilde\psi_\lambda,\\ \vspace{0.3cm}
\Rightarrow\, & u_\lambda\ \mbox{is a local}\ H^1(\Omega)\mbox{-minimizer of}\ \widetilde\psi_\lambda\ (\mbox{see Proposition \ref{Proposition 2}}).
\end{align*}
This proves the Claim.

We assume that $K_{\psi_{\lambda}}$ is finite or otherwise we already have a sequence of distinct positive solutions, all strictly bigger than $u_\lambda$ (note that $K_{\widetilde\psi_\lambda}\subseteq[u_\lambda)=\left\{u\in H^1(\Omega):\, u_\lambda(z)\leq u(z)\right.$ for a.a. $\left. z\in\Omega\right\}$) and so we are done. Using the Claim, we can find $\rho\in(0,1)$ small such that
\begin{eqnarray}\label{79}
\widetilde\psi_\lambda(u_\lambda)<\inf\left[\widetilde\psi_\lambda(u):\, \|u-u_\lambda\|=\rho\right]=\widetilde m_\rho^{\lambda}.
\end{eqnarray}
Hypothesis $H_6$(ii) implies that
\begin{eqnarray}\label{80}
\widetilde\psi_\lambda(t \hat u_1)\rightarrow-\infty\ \mbox{as}\ t\rightarrow+\infty.
\end{eqnarray}
Then \eqref{76}, \eqref{79}, \eqref{80} permit the use of Theorem \ref{Theorem 1} (the mountain pass theorem). Therefore we can find $\hat u_\lambda\in H^1(\Omega)$ such that
\begin{align*}
& \hat u_\lambda\in K_{\widetilde\psi_\lambda}\subseteq[u_\lambda)\ \mbox{and}\ \tilde m_\rho^\lambda\leq\widetilde\psi_\lambda(\hat u_\lambda),\\ \vspace{0.3cm}
\Rightarrow\,  & \hat u_\lambda\in S(\lambda)\subseteq\textrm{int}\, C_+\ (\mbox{see \eqref{75}}), u_\lambda\leq\hat{u}_\lambda,\, u_\lambda\neq\hat u_\lambda.
\end{align*}
Moreover, if $\rho=\|\hat u_\lambda\|_\infty$ and $\hat\xi_\rho>0$ is as postulated by hypothesis $H_6$(iv), then
\begin{align*}
& -\Delta u_\lambda(z)+(\xi(z)+\hat\xi_\rho) u_\lambda(z)\\ \vspace{0.3cm}
& {}\hspace{0.3cm} = (\lambda+\hat\xi_\rho)u_\lambda(z)+f(z,u_\lambda(z))\\ \vspace{0.3cm}
& {}\hspace{0.3cm} \leq(\lambda+\hat\xi_\rho)\hat u_\lambda(z)+f(z,\hat u_\lambda(z))\\ \vspace{0.3cm}
& \hspace{3cm}(\mbox{recall}\ u_\lambda\leq \hat u_\lambda \ \mbox{and see hypothesis}\ H_6\mbox{(iv)}) \\ \vspace{0.3cm}
& {}\hspace{0.3cm} = -\Delta\hat u_\lambda(z)+(\xi(z)+\hat\xi_\rho)\hat u_\lambda(z)\ \mbox{for a.a.}\ z\in\Omega, %\\
\end{align*}\begin{align*} \vspace{0.3cm}
\Rightarrow\, & \Delta(\hat u_\lambda-u_\lambda)(z)\leq(\|\xi^+\|_\infty+\hat\xi_\rho)(\hat u_\lambda-u_\lambda)(z)\ \mbox{for a.a.}\ z\in\Omega\\ \vspace{0.3cm}
& \hspace{3cm}(\mbox{see hypothesis}\ H(\xi)')\\ \vspace{0.3cm}
\Rightarrow\, & \hat u_\lambda-u_\lambda\in\textrm{int}\, C_+\ (\mbox{by the strong maximum principle}).
\end{align*}
\end{proof}
Summarizing the situation for problem ($P_\lambda$) when the perturbation $f(z,\cdot)$ is superlinear, we can state the following theorem.

\begin{theorem}\label{Theorem 21}
If hypotheses $H(\xi)'$, $H(\beta)$, $H_{6}$ hold, then
\begin{enumerate}
\item [(a)] for all $\lambda\geq\widehat\lambda_1$ problem ($P_\lambda$) has no positive solution (that is, $S(\lambda)=\emptyset$);
\item [(b)] for every $\lambda\in\mathcal{L}=(-\infty,\widehat\lambda_1)$ problem ($P_\lambda$) has at least two positive solutions $u_\lambda$, $\hat u_\lambda\in\textrm{int}\, C_+$, $\hat u_\lambda-u_\lambda\in\textrm{int}\, C_+$;
\item [(c)] for every $\lambda\in\mathcal{L}=(-\infty,\widehat\lambda_1)$ problem ($P_\lambda$) has a smallest positive solution $\bar u_\lambda$ and the map $\lambda\mapsto\bar u_\lambda$ from $\mathcal{L}=(-\infty,\widehat\lambda_1)$ is strictly increasing and left continuous.
\end{enumerate}
\end{theorem}

%For acknowledgements section, please don't number the section, please begin it with \section*{Acknowledgements}
\section*{Acknowledgments} The authors wish to thank
a knowledgeable referee for remarks which helped them to improve the presentation.
This research was supported by the Slovenian Research Agency grants P1-0292, J1-7025, and J1-6721, 
and the  Romanian National Authority for Scientific
Research and Innovation, CNCS-UEFISCDI grant
PN-II-PT-PCCA-2013-4-0614.
% You may incorporate your references as follows in your main tex file.
% Using BibTex is not recommended but can be handled.

\bigskip

\end{document}